\documentclass[11pt,a4paper]{report}
\usepackage{graphicx}
\usepackage{epsf}
\usepackage{yhmath}
\usepackage{amssymb,amsmath}
\usepackage{mathrsfs}

\usepackage{agdiss}

\renewcommand{\geq}{\geqslant}
\renewcommand{\leq}{\leqslant}

\newtheorem{theorem}{Theorem}[chapter]
\newtheorem{cor}[theorem]{Corollary}
\newtheorem{lemma}[theorem]{Lemma}
\newtheorem{conj}[theorem]{Conjecture}
\newtheorem{prop}[theorem]{Proposition}

\theoremstyle{remark}

\theoremstyle{definition}
\newtheorem{defi}[theorem]{Definition}

\newcommand{\C}{\mathbb{C}}

\newcommand{\R}{\mathbb{R}}

\newcommand{\la}{\langle}
\newcommand{\ra}{\rangle}
\newcommand{\e}{\varepsilon}
\newcommand{\B}{\mathfrak{B}}
\newcommand{\h}{\mathscr{H}}
\newcommand{\conv}{\mathrm{conv}}
\newcommand{\K}{\mathbb{K}}
\newcommand{\rk}{\mathrm{rk \,}}
\newcommand{\ee}{\mathcal{E}}
\newcommand{\hh}{\mathcal{H}}
\newcommand{\1}{\mathbf 1}
\newcommand{\tr}{\mathrm{tr \,}}
\newcommand{\pos}{\mathrm{pos \,}}
\newcommand{\D}{\mathscr{D}}

\newcommand{\al}{\alpha}
\newcommand{\be}{\beta}
\newcommand{\g}{\gamma}
\newcommand{\G}{\Gamma}

\newcommand{\La}{\Lambda}

\newcommand{\eps}{\varepsilon}
\newcommand{\E}{\mathbb{E}}
\newcommand{\Z}{{\mathbb{Z}}}
\newcommand{\PP}{\mathbb{P}}
\newcommand{\A}{\mathrm{A}}
\newcommand{\cc}{\mathcal{C}}
\newcommand{\dist}{\textrm{dist}}

\newcommand{\diag}{\mathrm{diag}}
\newcommand{\inte}{\mathrm{int}}

\title{Analytic and Probabilistic Problems in Discrete Geometry}
\author{Gergely Ambrus}

\begin{document}
\TitlePage \decpage

\abstract The thesis concentrates on two problems in discrete geometry, whose
solutions are obtained by analytic, probabilistic and combinatoric tools.

The first chapter deals with the strong polarization problem. This states that
for any sequence $u_1,\dots,  u_n$ of norm 1 vectors in a real Hilbert space
$\mathscr H$, there exists a unit vector  $v \in \mathscr H$, such that
\[
\sum \frac{1}{\la u_i, v \ra^2} \leq n^2.
\]
The 2-dimensional case is proved by complex analytic methods. For the higher
dimensional extremal cases, we prove a tensorisation result that is similar to
F. John's theorem about characterisation of ellipsoids of maximal volume. From
this, we deduce that the only full dimensional locally extremal system is the
orthonormal system. We also obtain the same result for the weaker, original
polarization problem.

The second chapter investigates a problem in probabilistic geometry. Take $n$
independent, uniform random points in a triangle $T$. Convex chains between two
fixed vertices of $T$ are defined naturally. Let $L_n$ denote the maximal size
of a convex chain. We prove that the expectation of $L_n$ is asymptotically
$\alpha \, n^{1/3}$, where $\alpha$ is a constant between $1.5$ and $3.5$ -- we
conjecture that the correct value is 3. We also prove strong concentration
results for $L_n$, which, in turn, imply a limit shape result for the longest
convex chains.

\endabstract

\acknowledgements

I would like to express my deep gratitude towards my supervisors, Keith M. Ball
and Imre B\'ar\'any. Our -- luckily, innumerable --  meetings always served as a
constant inspiration to me. While they have been the greatest source of
knowledge on mathematical research,  I also have learnt much from them about
the general philosophy of this profession. Had it been not for them, I may not
have become a mathematician.

Many thanks go to my former supervisors, Ferenc Fodor and Andr\'as Bezdek, who,
besides helping to start my research carrier, also played a large part in
shaping my young mind.

The generous support of the UCL Graduate School Research Scholarship was
essential for living in London.  That this time was as enjoyable as it has been
is largely due to my flatmates and friends, in particular, members of the
Imperial College Caving Club.

Finally, I thank my family members Judit, Imre and P\'eter for the warm and
encouraging atmosphere.
\newpage

\foreword The two topics discussed in this thesis are of a quite different
character. Chapter~1 is concerned with functional analytic properties of
discrete point sets. Chapter~2 belongs to the area of probabilistic discrete
geometry, and it reflects a more quantitative approach. This difference is by
virtue of my having two supervisors. However, in all the subsequent results,
the main motivating force is the underlying, clear and beautiful geometric
structure, that provides a  natural bond of the dissertation.
\newpage

\tableofcontents
\chapter{Polarization problems} \label{polarchapter}

The original polarization problem states the following: for any sequence $u_1,
\dots ,u_n$ of unit vectors in $\R^n$, there exists a unit vector $v \in \R^n$,
for which
\[
\prod |\la u_i, v \ra| \geq {n^{-n/2}}.
\]
We will also study the following stronger conjecture, that we call the strong
polarization problem. This asserts that under the above conditions, there is a
a unit vector $v$, such that
\[
\sum \frac{1}{\la u_i, v \ra^2 } \leq n^2.
\]
After giving a picture of the state of the art of the problem, in
Sections~\ref{complextoolssec} and~\ref{planarcasesec} we give a complex
analytic proof for the strong polarization problem in the case, when all the
vectors are in a plane. The proof depends on the structure of equioscillating
functions. For the higher dimensional problem, by linear algebraic
transformations described in Section~\ref{linalgsec}, we arrive to conjectures
about the location of inverse eigenvectors of Gram matrices. This is followed
by a geometric interpretation, where the difference between the two conjectures
becomes apparent as well. In Section~\ref{polarsec}, by an argument similar to
F. John's theorem, we deduce that the only full dimensional extremal vector
system for the polarization problem is the orthonormal system. Finally, in
Section~\ref{strongpolarsec}, we prove the analogous statement for the strong
polarization problem, and we characterise the locally extremal cases,
regardless of their dimension.

\section{History}\label{historypolarsec}

Polarization problems originate from the theory of infinite dimensional Banach
spaces. The reason for the term is that they are relatives of the general
polarization inequality, which relates the norm of a homogeneous polynomial to
the norm of its associated symmetric linear form. This inequality and other
related topics can be found in the monograph of Dineen \cite{dineen}, see
Section 1.3 therein.

The first articles about the polarization problem have been published roughly
at the same time, in 1998, by Ryan and Turett \cite{ryanturett} and by Ben\'itez,
Sarantopoulos and Tonge \cite{benitez}. They introduced the following notion.

\begin{defi}[\cite{benitez}]\label{polarconstdef}
Let $X$ be a Banach space and $X^*$ its dual space. The {\em $n^{th}$ linear
polarization constant of $X$}, to be denoted by $c_n(X)$, is given by
\[
c_n(X)=\inf \{M>0: \|\phi_1\| \dots \| \phi_n \| \leq M \|\phi_1 \dots \phi_n
\|,\ \forall \phi_1, \dots, \phi_n \in X^* \}.
\]
The {\em polarization constant of $X$} is
\[
c(X) = \limsup_{n \rightarrow \infty} \, (c_n(X))^{1/n}.
\]
\end{defi}

Ryan and Turett investigate the geometric structure of spaces of polynomials
and their preduals, and they prove that $c_n(X)< \infty$. The paper
\cite{benitez} is devoted to the polarization constant, and among more general
results, the authors show that for complex Banach spaces, $c_n(X) \leq n^n$.
For Banach spaces in general, this is the best possible upper bound, as is
shown by choosing $X=l_1$, and $\phi_i$ to be different coordinate functionals.
On the other hand, for arbitrary spaces, we have the trivial lower bound
$c_n(X) \geq 1$.

R\'ev\'esz and Sarantopoulos showed \cite{reveszsaran} that in
Definition~\ref{polarconstdef}, the ``$\limsup$'' can be changed to ``$\lim$''.

For real Banach spaces, K. Ball's affine plank theorem \cite{ballsymm}
 applies, and it yields a stronger result: for any set
of functionals $\phi_1, \dots, \phi_n$ in $X^*$, there is a point $x \in B_X$,
such that $|\phi_i(x)\| \geq  \|\phi_i\|/n$ {\em for every $i$}. Thus, for real
and complex Banach spaces the same result holds: $c_n(X) \leq n^n$.

The next stage was investigating Hilbert spaces. Let $\h$ be a (real or
complex) Hilbert space. By the Riesz Representation Theorem,  elements of
$\h^*$ are obtained by taking inner products with elements of $\h$. Therefore,
if $S_\h$ denotes, as usual, the unit sphere of $\h$, then
\[
c_n(\h) = \inf \left\{M>0: \forall \, u_1, \dots, u_n \in S_\h, \exists \, v
\in S_\h: |\la u_1, v\ra \dots \la u_n, v\ra| \geq \frac{1}{M} \right\}.
\]
The statement means that for any set of $n$ unit vectors in $\h$, there exists
a unit vector which is ``far away'' from subspaces orthogonal to the given
vectors. Considering an orthonormal system $u_1, \dots, u_n$, the inequality
between the geometric and the quadratic means implies that if $\dim \h \geq n$,
then for any unit vector $v$,
\[
\left| \prod \la u_i, v \ra  \right| \leq \frac{1}{n^{n/2}} \sum \la u_i, v
\ra^2 = \frac{1}{n^{n/2}},
\]
and hence $c_n(\h)\geq n^{n/2}$. On the other hand, using Dvoretzky's theorem,
it is not hard to show that if $X$ is an infinite dimensional Banach space,
then $c_n(X) \geq c_n(l_2^n)$, where $l_2^n$ is $\C^n$ endowed with the $l_2$
norm. Either by this result, or from the complex version of Bang's Lemma, it
follows that $c_n(\h) \leq n^n$.

It is natural to conjecture that the ``worst'' case arises when $(u_i)_1^n$ is
the $n$-dimensional orthonormal system: one would think that the orthogonal
subspaces of the vectors are ``spread out'' the most in this case.
Arias-de-Reyna proved in 1998 \cite{arias}, that for {\em complex Hilbert
spaces}, indeed, the right constant is $n^{n/2}$, and $c_n(\h)=n^{n/2}$, if
$\h$ is at least $n$ dimensional. His pretty proof is based on estimating the
variance of products of complex Gaussian random variables with the aid of
Lieb's inequality on permanents. He also conjectured that, as in the case of
Banach spaces, the best possible constant for real Hilbert spaces agrees with
the one for complex Hilbert spaces. Assuming that the dimension of the space is
at least $n$, and that $v$ is in the subspace spanned by $u_1, \dots u_n$, the
statement goes as follows.

\begin{conj}[Real polarization problem]\label{realpolconj}
For any collection $u_1, \dots ,u_n$ of unit vectors in $\R^n$, there exists a
unit vector $v \in \R^n$, such that
\begin{equation}\label{realpolineq}
\prod_{i=1}^{n} |\la u_i, v \ra| \geq {n^{-n/2}}.
\end{equation}
\end{conj}

Informally, the conjecture says that for any system of unit vectors in $\R^n$,
there is a unit vector that has ``large'' inner product with them in the above
sense. We will see that it cannot be required that the all the inner products
are large, unlike in the case of the plank problems.

As a converse of this statement, it {\em is } true, and a well-known fact, that
there is a unit vector $v$, for which $|\la u_i, v \ra|\leq 1/\sqrt{n}$ for all
$i$. For a generalisation of this, see Ball and Prodromou~\cite{ballprod}.

The complex plank theorem of K. Ball, published in 2001 \cite{ballcomplex},
states that if $u_1, \dots, u_n$ are unit vectors in a complex Hilbert space
$\h$, and $(t_i)_1^n$ is a sequence of positive reals satisfying $\sum t_i^2
=1$, then there exists another unit vector $v \in \h$, for which $|\la u_i, v
\ra| \geq t_i $ for every $i$.  On one hand, it immediately implies
Arias-de-Reyna's  estimate for the polarization constant of complex Hilbert
spaces. On the other hand, the result of Ben\'itez, Sarantopoulos and Tonge for
complex Banach spaces also follows from it. To this end, let $\phi_1, \dots,
\phi_n \in X^*$, and for each $i$, let $x_i$ be a point in $B_X$ where $\phi_i$
attains its norm. We shall search for a point $x$ in $\textrm{span} \{x_1,
\dots, x_n \}$, where $| \phi_1(x) \dots \phi_n(x)|$ is large. Hence we may
assume that $X$ is an $n$-dimensional Banach space. If $X$ and $Y$ are
isomorphic Banach spaces, then their {\em Banach--Mazur distance} $d(X,Y)$ is
given by
\[
d(X,Y) = \inf \{  \|T \| \| T^{-1}\|: \  T:X \rightarrow Y \textrm{ is an
isomorphism} \}
\]
Now, the well-known result of F. John \cite{john} about characterization of
simplices of maximal volume in convex bodies implies that if $X$ is an
$n$-dimensional Banach space, then $d(X, l_2^n) \leq \sqrt{n}$. Applying the
complex plank theorem, it easily follows (see \cite{reveszsaran}), that there
exists a point $x$ in the unit ball of $X$, for which $|\phi_i(x)| \geq 1/n$
for every $i$, which, in turn, implies that $c_n(X) \leq n^n$ for any complex
Banach space.

The real polarization problem has been investigated in many articles. The
complex result applied to the natural complexification of $\R^n$ yields that
$c_n(\R^n) \leq 2^{n/2-1} n^{n/2}$ (this was already mentioned in
\cite{ryanturett}, see also R\'ev\'esz and Sarantopoulos \cite{reveszsaran}).
Pappas and R\'ev\'esz proved in \cite{pappas} the following result: if $\K$ denotes
$\C$ or $\R$, then
\[
c(\K^d) = e^{- L(d, \K)},
\]
with
\[
L(d,K) = \int_{S_{\K^d}}\log |\la x, u \ra|\,  d \sigma(x),
\]
where $u$ is an arbitrary vector of $S_{\K^d}$  and $\sigma$ denotes the
normalised surface area measure.

It turns out that if the number of dimensions is at most $5$, then
Conjecture~\ref{realpolconj} can be proved by choosing a unit vector $v$ which
is obtained by normalising one point of the Bang system $\B$ generated by $u_1,
\dots, u_n$ (see Pappas and R\'ev\'esz \cite{pappas}). Matolcsi and Mu\~{n}oz showed
\cite{matomuno}, that this approach fails to prove the general conjecture in
higher dimensions; as a positive result, they managed to derive from it that
the orthonormal system is locally extremal with respect to the polarization
problem.

Another approach is to relate the best constant in (\ref{realpolineq}) to
eigenvalues and the determinant of the Gram matrix of the vector system $( u_1,
\dots, u_n )$, using a method that is similar to the one presented in
Section~\ref{linalgsec}. This idea has been raised by Marcus (see
\cite{reveszsaran}), and later elaborated by Matolcsi (\cite{mato1} and
\cite{mato2}). However, due to the difficulties of estimating the various
quantities related to the eigenvalues, the resulting inequalities do not seem
to be more approachable than the original one.

In $2008$, P. Frenkel \cite{frenkel}  returned to the method of Arias-de-Reyna
used for the case of complex Hilbert spaces. He managed to strengthen
Hadamard's inequality on determinants and Lieb's inequality on permanents with
the aid of pfaffians and hafnians. These results led to the following bound:
\[
c_n(\R^d) \leq \sqrt{n(n+2)(n+4) \dots (3n - 2)}< \left( \frac{3 \sqrt{3}}{e}\,
n \right)^{n/2} \approx (1.91)^{n/2}n^{n/2}
\]
At the moment, this is the strongest general bound on the real polarization
constant.

Also in 2008, Leung, Li, and Rakesh proved that if Conjecture~\ref{realpolconj}
fails, then the minimising vector system $(u_1, \dots, u_n)$ must be linearly
dependent. Their approach is similar to the one in Section~\ref{linalgsec}.

It was observed by P. Frenkel and K. Ball, that the following, stronger
alternative of the polarization problem has remarkable geometric properties. We
will mainly devote our attention to this problem.

\begin{conj}[Strong polarization problem]\label{strongpolar}
For any set $u_1, \dots ,u_n$ of unit vectors in $\R^d$, there exists a unit
vector $v \in \R^d$, such that
\[
\sum_{i=1}^{n} \frac{1}{\la u_i, v \ra^2 } \leq n^2.
\]
\end{conj}
By the arithmetic mean--geometric mean inequality, we immediately see that the
strong polarization problem is indeed stronger than
Conjecture~\ref{realpolconj}, the real polarization problem. The advantage of
this version over the older one will become apparent in the subsequent
sections. For illustration, let us present one aspect here.

It is conjectured that the only  extremal vector system in the real
polarization conjecture is the orthonormal system consisting of $n$ unit
vectors in $\R^n$. Therefore, if the number of vectors is larger than the
dimension of $X$, we expect a stronger inequality to hold. The simplest example
of this phenomenon is obtained when $X= \R^2$: If $(u_1, \dots, u_n)$ be a
system of vectors on the unit circle, then, via the connection to the Chebyshev
constant, the best constant turns out to be $2^{n-1}$, see~\cite{anagn}. This
is obtained when the point set $(u_1, -u_1, \dots, u_n, -u_n)$ is equally
distributed on the unit circle. The same example shows as well that the
assertion of the affine plank theorem is essentially sharp, and nothing close
to the estimate of the complex plank problem is true in the real setting.

Considering the strong polarization problem, the picture is entirely different.
As will be proved in Section~\ref{planarcasesec}, the best constant obtained
for systems of $n$ vectors on the unit circle is the same as the one we get for
the $n$-dimensional orthonormal system! Therefore, we ``don't gain anything''
by leaving the 2-dimensional space for $\R^n$, although, intuitively, one would
think that in the latter it is possible to go ``much farther away'' from the
orthogonal subspaces than in the plane. This rather remarkable geometric
property was the first to suggest that the strong polarization problem is a
good deal more natural than its original version, and in some sense it serves
as the real analogue of the complex plank problem.

\section{Complex analytic tools}\label{complextoolssec}

The planar, $d=2$ case of the polarization problems can perhaps be most
naturally formulated on the complex unit circle $T$, that we sometimes identify
with the interval $[0,2 \pi]$ via the formula $z = e^{i t}$. Suppose that the
norm 1 vectors $u_1, \dots,u_n$ on $S^1$ are given by
\[
u_j = \left(\cos \frac{t_j}{2}, \sin \frac{t_j}{2}\right).
\]
We shall search for the vector $v \in S^1$ in the form
\[
v =
\left(\cos\left(\frac{t}{2}-\frac{\pi}{2}\right),\sin\left(\frac{t}{2}-\frac{\pi}{2}\right)\right).
\]
Define the complex numbers $z_j$ and $z$ on $T$ by
\[
z_j = e^{i t_j} \, , \, z= e^{i t}.
\]
Then, with the above notations,
\begin{equation}\label{sinabs}
\la u_j, v \ra = \sin \left(\frac{t-t_j}{2}\right) = \frac{|z - z_j|}{2}\,.
\end{equation}
Thus, the $d=2$ case of the polarization problems can be formulated as
statements about trigonometric polynomials (see the definition below). It is
natural, and indeed fruitful, to consider the analytic continuation of these
functions from $T$ to the complex plane, resulting in complex rational
functions. By this means, we derive alternate formulations of the original
statements that can be tackled by strong complex analytic tools. For an
illustration of the power of this method, let us mention one example.

As we have discussed earlier, it is conjectured that the only extremal vector
system for the original polarization problem is the $n$-dimensional orthonormal
system. Therefore, if all the vectors $(u_i)_1^n$ are on the plane, we expect a
stronger inequality to hold. The following statement gives the estimate that is
the best possible.

\begin{prop}[\cite{anagn}]
For any set $u_1, \dots, u_n$ of unit vectors on $S^1$, there exists $v \in
S^1$, such that
\[
\prod |\la u_j,v \ra| \geq 2^{-(n-1)}.
\]
\end{prop}

\begin{proof}
Using (\ref{sinabs}), it suffices to prove that for any set $z_1, \dots , z_n$
of complex numbers of norm 1, there exists $z \in T$, for which
\[
|(z-z_1) \dots (z-z_n)|\geq 2.
\]
Define the complex polynomial $Q(z):=\prod(z-z_j)$. Then for any complex number
$w \in T$, we have
\[
\frac{1}{n}\sum_{k=1}^{n} Q(w \, e^{i 2 \pi k/n}) = w^n + (-1)^n z_1 \dots z_n.
\]
Choose $w$ so that $w^n=(-1)^n z_1 \dots z_n$. Then, by the above formula,
\[ 2= |w^n + (-1)^n z_1 \dots z_n| = \left|\frac{1}{n}\sum_{k=1}^{n} Q(w \, e^{i 2 \pi k/n})   \right|
\leq\frac{1}{n}\sum_{k=1}^{n} |Q(w \, e^{i 2 \pi k/n})|.
\]
Therefore there exists a $k$, for which $|Q(w \, e^{i 2 \pi k/n})|\geq 2$.
Also, if we take $z_j=e^{i 2 \pi j/n}$, then it is easy to see that the
estimate is sharp.
\end{proof}

We note that the quantity
\[
M_n(S^1)=\inf_{x_1, \dots x_n \in S^1} \sup_{x \in S^1} \|x-x_1\|\dots
\|x-x_n\|.
\]
is called the {\em $n^{\textrm{th}}$ Chebyshev constant} of the unit circle.
Also, the statement implies that the polarization constant of $\R^2$ is $2$.
The same result for $\C^2$ can be obtained by a similar approach \cite{anagn}.

For the planar case of the strong polarization problem, we do not know such a
simple proof as the one above. Still, a complex analytic proof can be achieved,
which will be presented in Section~\ref{planarcasesec}. For convenience, we
establish the necessary complex analytic tools in the present chapter.

Some of the following results had been proved in the early $20^{\mathrm{th}}$
century in connection with the theory of orthogonal polynomials, and the others
are of a similar spirit as well. In definitions, we mostly follow the
manuscripts of Szeg\H o \cite{szego} and  P\'olya and Szeg\H o \cite{polyaszego}.

A {\em complex polynomial} is a polynomial with complex coefficients. The
quotient of two complex polynomials is called a {\em (complex) rational
function}. A {\em trigonometric polynomial of degree $n$ }  is a $2
\pi$-periodic function defined on the real line given by
\[ f(t)= a_0 + a_1 \cos t + b_1
\sin t + \dots + a_n \cos (n t) + b_n \sin(n t),
\]
where the coefficients are real numbers. We mention that sometimes the
coefficients are allowed to be arbitrary complex numbers, however, we do not
need this generality. Also, via the formula $z=e^{i t}$, a trigonometric
polynomial can be understood as a function defined on $T$.

Any trigonometric polynomial of degree $n$ can also be written in the form
\[
f(t)=\sum_{j=-n}^n \alpha_j e^{i n t} = \frac{p(z)}{z^n},
\]
where $z= e^{i t}$ and $p(z)$ is a polynomial of degree $2n$. In particular,
$f(t)$ cannot have more than $2n$ zeroes on the interval $[0,2 \pi ]$.
Moreover, since all the coefficients of $f(t)$ are real, in the above
representation $\alpha_j = \bar{\alpha}_{-j}$ holds for every $j$. This
property turns out to be of special importance in view of the following
definition \cite{szego}.

\begin{defi}\label{inverspol}
Let $g(z)= a_0 + a_1 z + \dots + a_n z^n$ be a complex polynomial. Its {\em
reciprocal polynomial of order $n$} is defined by
\[
g^*(z)=  \bar{a}_n + \bar{a}_{n-1}z + \dots +\bar{a}_0 z^n.
\]
\end{defi}
It is easy to see that $\overline{g^*(z)} = g(1/z)z^n$. Note that we do not
require $a_n \neq 0$, and hence if $g(z)$ has precise degree $n$, then its
reciprocal polynomials can be defined of any order at least $n$. However, if we
do not specify otherwise, the order will always be the precise degree of
$g(z)$.

For any non-zero complex number $z$, let $z^*$ denote its image under the
inversion with respect to complex unit circle $T$:
\[
z^* = \frac{1}{\bar z}.
\]

It is easy to see that if the non-zero roots of $g(z)$ are $\alpha_1, \dots ,
\alpha_k$, then the non-zero roots of $g^*(z)$ are $\alpha^*_1, \dots ,
\alpha^*_k$. Moreover, if $|z|=1$, then $\bar z = 1/z$, and therefore
\begin{equation}\label{inverst}
g^*(z)= z^n \overline{g(z)},
\end{equation}
consequently, $|g^*(z)|= |g(z)|$. Since $z^* = z$ for any $z \in T$, the roots
of $g(z)$ and $g^*(z)$ agree. Therefore we immediately obtain

\begin{lemma}\label{circlepol}
If all zeroes of the complex polynomial $g(z)$ have modulus 1, then
\[g^*(z) = \gamma g(z)\] for a complex constant $\gamma$ with $|\gamma|=1$.
\end{lemma}

Now, if $f(t)$ is a trigonometric polynomial of degree $n$, then  $f$ can be
written as
\[
f(t)=\frac{p(e^{i t })}{e^{i n t}}
\]
where $p(z)$ is a complex polynomial of degree $2n$ with $p(z)=p^*(z)$. It is
easy to see that this relation is, in fact, an equivalence (see
\cite{polyaszego}, Problem VI. 12). Equation (\ref{inverst}) now induces a
close connection between trigonometric polynomials and the real and imaginary
parts of arbitrary polynomials. Let $g(z)$ be a polynomial of degree $n$. If
$|z|=1$, then
\[
\Re g(z)= \frac 12 \left(g(z) + \overline{g(z)}\right) = \frac 12 \left(g(z) +
\frac{g^*(z)}{z^n}\right) = \frac{h(z)}{2 z^n},
\]
where $h(z)$ is a polynomial of degree $2n$ with $h(z)= h^*(z)$. Therefore the
real part of a polynomial of degree $n$ on the unit circle $T$ is  a
trigonometric polynomial of degree $n$. A similar argument yields that the
imaginary part can be represented in the same way.

It also follows that if $f(t)$ is a trigonometric polynomial of degree $n$,
then the set of zeroes of its holomorphic continuation $F(t)$ from $T$ to the
complex plane is invariant under the inversion  to $T$. Therefore, if
$\alpha_1, \dots, \alpha_m$ are the non-zero roots of $F(t)$ in the open unit
disc, then writing $g(z) = \prod (z- \alpha_j)$, $F(t)$ can be factorized as
\begin{equation}\label{postrig}
z^{-n} g(z) g^*(z) h(z),
\end{equation}
where $h(z)$  is a polynomial with zeroes only on $T$. Moreover, if $f(t)$ is
non-negative, then all the zeroes on $T$ are of even multiplicity, and
therefore $f(t)$ can be written as
\[
f(t)=|g(e^{it})|^2,
\]
where $g(z)$ is a polynomial of degree $n$. This is Fej\'er's representation
theorem, see Szeg\H o \cite{szego} 1.2.

The following observation is the converse of Lemma~\ref{circlepol}. It can be
found for example in the first edition of \cite{polyaszego}.
\begin{lemma}\label{unitzeros}
Suppose that the non-zero polynomial $g(z)$ has no zeroes in the open unit
disc. Then for any complex number $\gamma$ of modulus 1, all zeroes of $g(z) +
\gamma g^*(z)$ lie on the unit circle~$T$.
\end{lemma}

\begin{proof}
We may assume that $\gamma = 1$ and that $g(z)$ has no zeroes on $T$, therefore
$g(z)/g^*(z)$ maps the unit circle continuously onto itself. Since $g(z)$ has
no zeroes in the unit disc, the winding number of the curve $\{g(z): z \in T
\}$ with respect to the origin is $0$. By virtue of (\ref{inverst}), the
winding number of $g(z)/g^*(z)$ is $-n$. Therefore there are at least $n$
points on $T$, where $g(z)+g^*(z)=0$, and since it is a polynomial of degree
$n$, all of its zeroes have modulus 1.
\end{proof}

We will be interested in rational functions that possess an interesting
oscillation property. Bearing this in mind, we introduce the following concept.
The definition is slightly modified compared to that in \cite{glader}.

\begin{defi}\label{equioscill}
The real valued function $f$  on $T$ is {\em equioscillating of order $n$}, if
there are $2n$ points $w_1, w_2, \dots, w_{2n}$ on $T$ in this order, such that
\[
f(w_j) = (-1)^{j} \|f\|_T
\]
for every $j=1, \dots , 2n$, and $|f(z)|<\|f\|_T$ if $z \neq w_j$ for any $j$.
\end{defi}

Although equioscillation in general is not a very specific property (plainly,
any real valued function on $T$ whose level sets are finite has a shift which
is equioscillating of some order), equioscillation of a possible maximal order
is a strong condition. This becomes apparent in the context of rational
functions.

Suppose that $R(z)$ is a rational function, whose numerator is of degree $k$
and whose denominator has degree $l$; then the real and imaginary parts of
$R(z)$ are the quotients of two trigonometric polynomials of degrees $k$ and
$l$, and therefore $\Re (R(z))$ and $\Im( R(z))$ cannot be equioscillating of
order larger than $\max \{k,l \}$.  A characterization of those rational
functions whose real and imaginary parts are oscillating with this maximal
order was given by Glader and H\"ogn\"{a}s in 2000 \cite{glader}. In order to
formulate their result, we need the definition of Blaschke products.

\begin{defi}\label{blaschkedef}
A {\em finite Blaschke product of order $n$} is a rational function of the form
\begin{equation}\label{blaschke}
B(z)= \rho \,z^k \prod_{j=1}^{n-k} \frac{z- \alpha_j}{1- \bar{\alpha}_j z}\;,
\end{equation}
where $\rho,\alpha_1, \dots, \alpha_{n-k}$ are complex numbers with $|\rho|=1$
and $0<|\alpha_j|<1$.
\end{defi}

Clearly, the zeroes of the numerator and those of the denominator are images of
each other under the inversion with respect to $T$. Furthermore, $B(z)$ maps
the unit circle onto itself. Therefore, it can be written in the form
\begin{equation}\label{blaschkereflect}
B(z)= \gamma \frac{g(z)}{g^*(z)},
\end{equation}
where $|\gamma|=1$ and $g(z)$ is a polynomial of degree $n$. This is the
crucial property that we shall use later.

With Blaschke products in our arsenal, we can formulate the result about
maximally equioscillating rational functions.

\begin{theorem}[Glader, H\"ogn\"{a}s, \cite{glader}]\label{maxosc}
If $R(z)$ is a rational function with numerator and denominator degrees at most
$n$, and $\Re (R(z))$ and $\Im( R(z))$ are equioscillating functions on $T$ of
order $n$, then $R(z) = c \,B(z)$ or $ R(z) = c / B(z)$, where $c$ is a real
constant and $B(z)$ is a finite Blaschke product of order $n$.
\end{theorem}

This essentially means that Blaschke products are in some sense the complex
analogues of Chebyshev polynomials.

The proof of Theorem~\ref{maxosc} involves several combinatorial steps, most of
which boil down to counting zeroes of rational functions. By taking advantage
of the perspective of reciprocal polynomials, we give an alternative proof for
the main constructive lemma. This result will serve as the crux of our proof
for the planar case  of Theorem~\ref{strongpolar}.

\begin{lemma}\label{equifunc}
Suppose that $1 = w_1, w_2, \dots, w_{2n}$ are different points on $T$ in this
order. Let $w$ be a point on $T$ different from each $w_j$. Then there exists a
complex polynomial $h(z)$ of degree $n$, such that
\[
\frac{h(w_k)}{h^*(w_k)}=(-1)^{k+1}
\]
for each $k = 1 ,\dots, 2n$, and
\[
\frac{h(w)}{h^*(w)} = i.
\]
\end{lemma}

\begin{proof}
Introduce the polynomials
\begin{eqnarray*}
g_1(z)&=& h(z)+h^*(z),\\
g_2(z)&=& h(z)-h^*(z).
\end{eqnarray*}
The original problem is equivalent to finding $g_1(z)$ and $g_2(z)$ with the
following properties:
\begin{itemize}
\item[(i)] The zeros of $g_1$ are $(w_{2k})$, where $1 \leq k \leq n$;
\item[(ii)] The zeros of $g_2$ are $(w_{2k-1})$, where $1 \leq k \leq n$;
\item[(iii)] $g_1(z) = g^*_1(z) $
\item[(iv)] $g_2(z) = - g^*_2(z) $
\item[(v)] $g_1(w) + i \, g_2(w)=0$.
\end{itemize}

\noindent In order to fulfill property (i), we search for $g_1(z)$ in the form
\begin{equation}\label{g11}
g_1(z) =  \alpha \prod_{k=1}^n{(z-w_{2k})},
\end{equation}
where $\alpha$ is a complex number of modulus 1. Lemma~\ref{circlepol} implies
that property (iii) is satisfied if the leading coefficient and the constant
term of $g_1(z)$ are conjugates of each other, that is,
\[
\bar \alpha = \alpha (-1)^n \prod w_{2k}.
\]
This is achieved by choosing $\alpha$ such that
\begin{equation}\label{alphaeq}
\alpha^2 =(-1)^n \prod \overline {w}_{2k}.
\end{equation}
\noindent

 Similarly, conditions (ii) and (iv) are fulfilled if $g_2(z)$ is defined
by
\[
g_2(z) = c \beta \prod_{k=1}^n{(z-w_{2k-1})},
\]
where $c$ is a non-zero real and $\beta$ is a complex number with $|\beta|=1$
satisfying
\begin{equation}\label{betaeq}
\beta^2 =(-1)^{n+1} \prod \overline {w}_{2k-1}.
\end{equation}

Using the fact that for any complex numbers $u \neq v$  of modulus 1,
\begin{equation}\label{arg}
\arg(u-v) \equiv \frac {\arg u + \arg v }{2} + \frac{\pi}{2}  \pmod \pi ,
\end{equation}
we obtain that if $\alpha$ and $\beta$ satisfy (\ref{alphaeq}) and
(\ref{betaeq}) respectively, then
\begin{equation}\label{argg1}
\arg g_1(z) \equiv  \frac {n \arg z}{2} \pmod \pi
\end{equation}
and
\begin{equation}\label{argg2}
\arg g_2(z) \equiv  \frac {n \arg z}{2} + \frac{\pi}{2} \pmod \pi
\end{equation}
 for any $z \in T$ (these also follow from (\ref{inverst}), since
 $\arg g(z) + \arg g^*(z) = n \arg z$). Thus
\[
\arg g_1(z) \equiv \arg (i\, g_2(z)) \pmod \pi
\]
on the unit circle. Since $w \in T$, $g_1(w) \neq 0 $ and $g_2(w) \neq 0$, $c$
can be chosen so that property (v) holds.
\end{proof}

\section{The planar case}\label{planarcasesec}
After preparing the complex analytic apparatus, the goal of this section is to
prove the $d=2$ case of Conjecture~\ref{strongpolar}. Referring to
(\ref{sinabs}), it can be stated in the complex setting as follows.

\begin{theorem}\label{planarpol}
For any set $z_1, \dots z_n$ of complex numbers of modulus 1, there exists a
complex number $z$ of norm 1, such that
\begin{equation}\label{planpolineq}
\sum \frac{1}{|z-z_j|^2} \leq \frac{n^2}{4}\;.
\end{equation}
\end{theorem}

First, we make use of the special structure of the function to be estimated.
For any sequence $(z_j)_1^n=\mathbf{z}\in T^n$, let
\begin{equation}\label{planarfunc}
G_{\mathbf{z}}(z)= \sum \frac{1}{|z-z_j|^2}
\end{equation}
be a function defined on $T$, and denote
\[
M(\mathbf{z}) = \min_{z \in T} G_\mathbf{z}(z).
\]
Our aim is to prove that
\begin{equation}\label{planarineq}
 M(\mathbf{z}) \leq n^2/4
\end{equation}
 for any
$\mathbf{z} \in T^n$.

Let $\mathcal{T}$ be the usual product topology on the space $T^n$. A sequence
$\mathbf{z}\in T^n$ is {\em locally extremal}, if there exists a neighbourhood
$\mathscr U$ of $\mathbf z$ in $\mathcal T$, such that for any $\mathbf {z'}\in
\mathscr U$,
\[
M(\mathbf{z}) \geq M(\mathbf{z'}).
\]
It clearly suffices to prove the inequality (\ref{planarineq}) for locally
extremal sets.

A real-valued function $g(z)$ defined on $T$ is called {\em convex}, if it is a
convex function of the argument of $z$. It is easy to see that $1/|z-z_j|^2$ is
convex on $T \setminus {z_j}$, and therefore $G_{\mathbf{z}}(z)$ is convex on
the arcs between the consecutive points of $\mathbf z$. Since $G_\mathbf{z}(z)$
has poles at each $z_j$, we obtain that it has exactly one local minimum on
each arc of $T$ between consecutive points of $\mathbf z$. (If two points of
$\mathbf z$ coincide, then the local minimum between them is defined to be
$\infty$.) For locally extremal sets, these minima follow a certain behaviour.
The information provided by the next lemma will make it possible to apply the
result of the previous section about equioscillating functions.

\begin{lemma}\label{equimini}
If $(z_j)_1^n=\mathbf z$ is a locally extremal set, then the local minima of
$G_{\mathbf{z}}(z)$ on the arcs of $T$ between consecutive points of $\mathbf
z$ are all equal.
\end{lemma}

\begin{proof}
Suppose on the contrary that $z_1 = e^{i t_1}$ and $z_2 = e^{i t_2}$ are two
consecutive points such that the local minimum of $G_{\mathbf z }(z)$ on the
arc $\wideparen{z_1 z_2}$ is strictly larger than $M(\mathbf{z})$; this also
implies that $z_1 \neq z_2$.  We can assume that $0 \leq t_1 \leq t_2 < 2\pi$
and that $\wideparen{z_1 z_2}$ is the set of points of $T$ with argument
between $t_1$ and $t_2$. Let $\e$ be a small positive number, and consider the
new set of points $\mathbf{z}'$ obtained from $\mathbf z$ by exchanging $z_1$
and $z_2$ for
\[
z_1' = e^{i (t_1-\e)} \; , \; z_2' = e^{i (t_2+\e)}.
\]
Let us compare the values of $G_{\mathbf z'}(z)$ to those of $G_{\mathbf
z}(z)$. First, suppose that $z\in \wideparen{z_1 z_2}$, where $z = e^{i t}$. By
symmetry, it suffices to consider the case $t_1 \leq t \leq (t_1 + t_2)/2$.
Then,
\[
\frac{1}{|z-z_1|^2}> \frac{1}{|z-z_1'|^2}\;,
\]
and furthermore, by convexity and symmetry,
\[
\left|\frac{1}{|z-z_1|^2}- \frac{1}{|z-z_1'|^2} \right| >
\left|\frac{1}{|z-z_2|^2}- \frac{1}{|z-z_2'|^2} \right|\;.
\]
Thus
\[
\frac{1}{|z-z_1|^2}+ \frac{1}{|z-z_2|^2} > \frac{1}{|z-z_1'|^2}+
\frac{1}{|z-z_2 '|^2}
\]
and hence
\[
G_{\mathbf z}(z)>G_{\mathbf z'}(z).
\]
Interchanging the roles of $z_j$ and $z_j'$ ($j=1,2$) yields that if $z\in
\wideparen{z_2' z_1'}$, then
\[
G_{\mathbf z}(z)<G_{\mathbf z'}(z).
\]
If $\e$ is sufficiently small, then the minimum of $G_{\mathbf z}(z)$ is
attained on the arc $\wideparen{z_2' z_1'}$, while the local minimum of
$G_{\mathbf z'}(z)$ on $\wideparen{z_1' z_2'}$ is still larger than the minimum
on $\wideparen{z_2' z_1'}$. Therefore
\[
M(\mathbf{z}) < M(\mathbf{z'}),
\]
which contradicts  the extremality of $(z_j)_1^n$.
\end{proof}

We note that Lemma~\ref{equimini} remains valid for any function instead of
$G_{\mathbf{z}}(z)$ that is obtained by taking the sum of translated copies of
a convex, axis-symmetric function on $T$ with one pole.

\begin{proof}[Proof of Theorem~\ref{planarpol}]\label{planarproof}
We may assume that $\mathbf{z} = (z_j)_1^n$ is a locally extremal set, and
therefore it necessarily consists of $n$ different points. Setting
\[ m = 2 \sqrt{M(\mathbf z)},
\]
the inequality (\ref{planarineq}), that we wish to prove,  is equivalent to the
statement $m \leq n$.

 For any $z$ and $z_j$ on $T$,
\begin{equation}\label{abspol}
|z - z_j|^2 = (z - z_j) \overline{(z-z_j)} = (z-z_j)\left(\frac 1z -\frac
1{z_j}\right) = - \frac{ (z- z_j)^2}{z \, z_j}.
\end{equation}
Thus, defining the rational function
\begin{equation}\label{rationalform}
R(z)  = \frac{\prod_{j=1}^{n}(z-z_j)^2}{-z\sum_{j=1}^{n} z_j \prod_{k \neq j}
(z -z_k)^2},
\end{equation}
we obtain by (\ref{planarfunc})  that $R(z) = 1/G_{\mathbf{z}}(z)$ for every
$z$ on $T$.

The degrees of the numerator and the denominator of $R(z)$ are $2n$ and at most
$2n - 1$, respectively. The zeroes are $(z_j)_1^n$ with multiplicity 2, and
$R(z)$ assigns real values on the unit circle. Moreover, Lemma~\ref{equimini}
implies that the function
\[
R(z)-\frac{2}{m^2},
\]
which is a rational function as well, oscillates equally between $-2/m^2$ and
$2/m^2$ of order $n$. Let $w_1, \dots, w_{2n}$ be the equioscillation points
such that $w_{2k}= z_k$ for every $k= 1, \dots, n$, and let $w$ be a further
point on $T$ satisfying $R(w) = 2/m^2$. Applying Lemma~\ref{equifunc} yields a
polynomial $h(z)$ of degree $n$, such that
\begin{equation}\label{rateq}
R(z)-\frac{2}{m^2} = \frac{2}{m^2}\, \Re\left(\frac{h(z)}{h^*(z)}\right)
\end{equation}
for every $z= w_1, \dots, w_{2n}, w$. Moreover, both functions assign real
values on $T$, and they have local extrema at the points $(w_j)_1^{2n}$,
therefore their derivatives vanish at these places.

Since $|h(z)|=|h^*(z)|$ on the unit circle,
\[
\frac{2}{m^2}+\frac{2}{m^2}\, \Re\left(\frac{h(z)}{h^*(z)}\right)=
\frac{1}{m^2}\left(2+\frac{h(z)}{h^*(z)}+\frac{h^*(z)}{h(z)}\right)=
\frac{(h(z)+h^*(z))^2}{m^2 \, h(z) h^*(z)}.
\]
Thus, from (\ref{rateq}) we deduce that the rational function
\[
R(z)-\frac{(h(z)+h^*(z))^2}{m^2 \, h(z) h^*(z)}
\]
has double zeroes at all the points $w_1, \dots, w_{2n}$, and it also vanishes
at $w$. On the other hand, its numerator is of degree at most $4n$. Hence, it
must be constantly 0, and using~(\ref{rationalform}), we obtain that
\begin{equation}\label{rateq2}
\frac{\prod_{j=1}^{n}(z-z_j)^2}{-z\sum_{j=1}^{n} z_j \prod_{k \neq j} (z
-z_k)^2}=\frac{(h(z)+h^*(z))^2}{m^2 \, h(z) h^*(z)}.
\end{equation}
In the rest of the proof, we investigate this equation; however, there is still
a fairly long way to go.

As in the proof of Lemma~\ref{equifunc}, we introduce the functions $g_1(z)=
h(z)+h^*(z)$ and $g_2(z)= h(z)-h^*(z)$. Then by (\ref{g11}) and
(\ref{alphaeq}),
\begin{equation*}
g_1(z) =  \alpha \prod_{j=1}^n{(z-z_{j})},
\end{equation*}
where $\alpha$ is a complex number of norm 1 satisfying
\begin{equation}\label{alphaeq2}
\alpha^2 =(-1)^n \prod \overline {z}_j\,.
\end{equation}
According to properties (iii) and (iv) of the proof of Lemma~\ref{equifunc}, $
g_1(z) = g^*_1(z)$ and $g_2(z) = - g^*_2(z)$,  hence they have the form
\begin{equation}\label{g1g2}
\begin{aligned}
g_1(z)&= \alpha z^n + \dots + \bar \alpha, \\
g_2(z)&= \beta  z^n + \dots - \bar \beta.
\end{aligned}
\end{equation}
Substituting $g_1(z)$ and $g_2(z)$, equation (\ref{rateq2}) transforms to
\begin{equation}\label{rateq3}
\frac{\prod_{j=1}^{n}(z-z_j)^2}{-z\sum_{j=1}^{n} z_j \prod_{k \neq j} (z
-z_k)^2} = \frac{g_1(z)^2}{\frac{m^2}{4} (g_1(z)^2-g_2(z)^2)}\,.
\end{equation}
Since the degree of the denominator on the left hand side is at most $2n-1$,
from (\ref{g1g2}) we deduce that
\begin{equation}\label{leadco}
\alpha = \pm \beta.
\end{equation}
The quotient of the leading coefficients of the numerators on the two sides of
(\ref{rateq3}), which is $\alpha^2$, is the same as the quotient of those of
the denominators. Therefore
\begin{equation*}\label{rateq4}
-\alpha^2 z\sum_{j=1}^{n} z_j \prod_{k \neq j} (z -z_k)^2 = \frac{m^2}{4}
(g_1(z)^2-g_2(z)^2).
\end{equation*}
Substituting $z = z_j$ and taking square roots yields
\[
\alpha \, z_j \prod_{k \neq j} (z_j -z_k) = \pm \frac{m}{2} \, g_2(z_j).
\]
Observe that this is equivalent to
\begin{equation}\label{dereq1}
z_j \, g_1'(z_j)=  \e_j \frac{m}{2} \, g_2(z_j),
\end{equation}
where $\e_j = \pm 1$.

Next, we show that $\e_j = \e_k$ for any $j$ and $k$. First, for any $j$,
\begin{align*}
\arg (g_1'(z_j)) &= \lim_{\delta \rightarrow 0+} \left(\arg \left(g_1(z_j e^{i
\delta})\right) - \arg (z_j e^{i \delta} - z_j)\right)\\ &= \lim_{\delta
\rightarrow 0+} \arg \left(g_1(z_j e^{i \delta})\right) - \arg z_j -
\frac{\pi}{2}
\end{align*}
and therefore
\begin{equation}\label{argeq1}
\arg (z_j g_1'(z_j)) = \lim_{\delta \rightarrow 0+} \arg \left(g_1(z_j e^{i
\delta})\right) - \frac{\pi}{2}\,.
\end{equation}
Second, from (\ref{argg1}) and (\ref{argg2}) it follows that
\[
\arg \frac {g_1(z)}{g_2(z)}  \equiv  \frac {\pi}{2} \pmod \pi
\]
on the unit circle. Since $g_1(z)$ and $g_2(z)$ are polynomials with single
zeroes only, their arguments change continuously on $T$ apart from their
zeroes, where a jump of $\pi$ occurs. It is easy to see that the zeroes of
$g_2(z)$ are the local minimum places of $G_{\mathbf z}(z)$, and therefore the
zeroes of $g_1(z)$ and $g_2(z)$ are alternating on $T$. This implies that
\[
\lim_{\delta \rightarrow 0+} \arg \frac{g_1(z_j e^{i \delta})}{g_2(z_j e^{i
\delta})}
\]
is the same for every $j$ modulo $2 \pi$. Now (\ref{argeq1}) yields that
\[ \arg \frac{z_j g_1'(z_j)}{g_2(z_j)}\]
is the same modulo $2 \pi$ for every $j$. A quick look at (\ref{dereq1})
reveals that, indeed, $\e_j$ is constant for all $j$. Let this constant be $\e
= \pm 1$.

From (\ref{dereq1}), we conclude that  the polynomial
\[
z \, g_1'(z) - \e \frac{m}{2} \, g_2(z)
\]
of degree $n$ attains 0 at all $(z_j)_1^n$, and hence its zeroes agree with
those of $g_1(z)$. Therefore there exists a complex number $\gamma$, such that
\[
z \, g_1'(z) - \e \frac{m}{2} \, g_2(z)  = \gamma \, g_1(z),
\]
and thus
\begin{equation}\label{rateq5}
 \e \frac{m}{2} \, g_2(z)  = z \, g_1'(z) - \gamma \, g_1(z).
\end{equation}
Equating the leading coefficients, referring to (\ref{g1g2}), gives
\begin{equation}\label{mgamma}
\e \frac{m}{2} \, \beta = (n - \gamma) \alpha,
\end{equation}
which, with the aid of (\ref{leadco}), yields that $\gamma \in \R$.

Finally, by comparing the leading coefficients and the constant terms in
(\ref{rateq5}) and using the form (\ref{g1g2}), we deduce that
$
(n- \gamma)\alpha = \gamma \alpha
$
and, since $\alpha \neq 0$,
\[
\gamma = \frac{n}{2}.
\]
Taking absolute values in (\ref{mgamma}) and bearing (\ref{leadco}) in mind, we
obtain that $m = n$, which is even stronger than the desired inequality in the
sense, that it shows that every locally extremal set is an extremal set.
\end{proof}

\section{Remarks about the complex proof}\label{remarkscomplexsec}

The proof of Theorem~\ref{planarpol} in the previous section does not give a
characterization of the extremal cases. However, in
Section~\ref{strongpolarsec}, by a different method, we shall prove that the
inequality (\ref{planpolineq}) is sharp only if there exists a complex number
$\rho$ of modulus 1 such that the set $\{ z_j/\rho \}$ equals to the set of
unity roots of order $n$. From this, via (\ref{rateq3}) and (\ref{argg1}), we
obtain that
\[
g_1(z) =  i \rho^{n/2}\left( \left(\frac{z}{\rho}\right)^n - 1 \right).
\]
Hence, by (\ref{rateq5}),
\[
g_2(z) =  i \rho^{n/2}\left( \left(\frac{z}{\rho}\right)^n + 1 \right).
\]
Therefore, again by (\ref{rateq3}) and by (\ref{abspol}), it follows that for
any $z \in T$,
\begin{equation*}\label{planarextr}
\sum \frac{1}{|z-z_j|^2}=\frac{1}{R(z)} = \frac{- n^2 \rho^n z^n}{(z^n -
\rho^n)^2}=\frac{n^2}{|z^n-\rho^n|^2}\;.
\end{equation*}
Setting $\rho =1$ and translating the result back into the real setting yields
the formula
\begin{equation}\label{cosform}
 \sum_{j=1}^n \left(\sin^2 \left(\frac{t}{2}-\frac{ j \pi}{n} \right)\right)^{-1}
  = \frac{2 n^2}{1-\cos n t}\;,
\end{equation}
that can also be obtained by comparing the values of the left hand side and its
derivative to those of the  Chebyshev polynomial at the points $t = j \pi /n$.
The special case of $t= \pi/n$  can be also be  derived from the  Riesz
Interpolation Formula, that was proved  in 1914 by Marcel Riesz \cite{riesz14}.
This states that for any trigonometric polynomial $f(t)$ of degree $n$, where
$n$ is even,
\[
f'(t)= \frac{1}{n}\sum_{j=1}^{n} (-1)^{j+1} \lambda_j f(t + t_j)
\]
with
\[
\lambda_j =\left(2 \sin^2 \left(\frac{t_j}{2} \right)\right)^{-1}\quad \textrm{
and } \quad t_j = \frac{(2j-1) \pi  }{n} \;.
\]
Setting $f(t)=\sin(n t/2)$ and $t=0$, we arrive at the desired case of
(\ref{cosform}):
\[
\sum_{j=1}^{n} \left(\sin^2\left( \frac{j \pi}{n} - \frac{\pi}{2n}\right)
\right)^{-1} = n^2.
\]

Our next remark concerns formula (\ref{rateq5}). The first question that comes
to one's mind is probably the following:
\bigskip

\noindent {\em For which polynomials $g(z)$ of degree $n$ does the polynomial
\begin{equation}\label{gderpol}
h(z) = z \, g'(z) - \frac{n}{2} \, g(z)
\end{equation}
 have zeroes only on the unit circle?}

\bigskip
At first glance, the question seems to be connected to Bernstein's inequality
about the derivatives of polynomials, the most relevant version of which reads
as follows: If $p(z)$ is a polynomial of degree $n$, then
\begin{equation*}
\| p'\|_T \leq n \| p\|_T.
\end{equation*}
Although this inequality is sharp (as shown by $p(z)=z^n$), under certain
constraints on the location of the zeroes of $p(z)$, a stronger statement
holds.   Erd\H os conjectured and  Lax proved (see \cite{lax}, and the article of
Erd\'elyi \cite{erdelyi}) that if $p$ has no zeroes in the open unit disc, then
\[
\|p'\|_D \leq \frac{n}{2} \|p\|_D,
\]
where $D$ is the closed unit disc. The factor $n/2$ is familiar from
(\ref{gderpol}). However, since the statement is about the maximum norms, it
cannot exclude the possibility of the existence of a zero of $h(z)$ inside the
unit circle.

In the present situation, we know that the zeroes of $g(z)$ all lie on $T$.
Quite surprisingly, this condition turns out to be sufficient.

\begin{prop}\label{unitzeroprop}
If $g(z)$ is a polynomial of degree $n$ with zeroes $z_1, \dots, z_n$ on the
unit circle, then all zeroes of $h(z) = z \, g'(z) - n/2 \, g(z)$ lie on the
unit circle as well.
\end{prop}
\begin{proof}
The following argument is based on the idea of P\'olya and Lax \cite{lax}. Let
\[
g(z) = c \prod_{j=1}^n (z-z_j),
\]
where $c \in \C$ and $|z_j|=1$ for all $j$. Let $w_1, \dots, w_n$ be points on
$T$ where $|g(z)|$ has local maxima on the unit circle. We shall show that the
polynomial $h(z)$ attains 0 at every point $w_j$.

By (\ref{arg}),
\[
\arg g(z)  \equiv \frac{n}{2} \arg z + \varphi \pmod \pi
\]
 with the constant
\[
\varphi = \arg c + \sum \frac{\arg z_j}{2} + \frac{n \pi}{2}.
\]
Therefore the function
\[
p(z) = e^{- i \varphi} \frac{g(z)}{z^{n/2}}
\]
is real for every $z \in T$. Moreover, $|p(z)| = |c|\, |g(z)|$ on the unit
circle, and hence $p'(w_j)=0$ for every $j$. Thus,
\begin{equation*}
h(w_j) = w_j \, g'(w_j) - \frac{n}{2} \, g(w_j) =e^{i \varphi}\frac{n}{2}
w_j^{n/2} p(w_j)  - \frac{n}{2}e^{i \varphi} w_j^{n/2} p(w_j)=0,
\end{equation*}
and hence the zeroes of $h(z)$ are $w_1, \dots, w_n$.
\end{proof}

Proposition~\ref{unitzeroprop} implies that starting from any polynomial
$g_1(z)$ of degree $n$ with zeroes on the unit circle, the rational function
$R(z)$ defined via  (\ref{rateq5}), (\ref{rateq3}) and (\ref{rationalform}) is
oscillating between $0$ and $4/n^2$ of order $n$ on the unit circle. In
Figure~\ref{complex1}, we illustrate $1/G_{\mathbf z}(z)$ and the
equioscillating $R(z)$ derived from $g_1(z)$, with the choice of
\begin{equation}\label{points}
z_1=1, \ z_2=e^{i 2\pi/5},\ z_3=i,\ z_4 = e^{i 3 \pi/4},\ z_5 = e^{i 5 \pi/4}.
\end{equation}

\begin{figure}[h]
\epsfxsize = 9.5cm \centerline{\epsffile{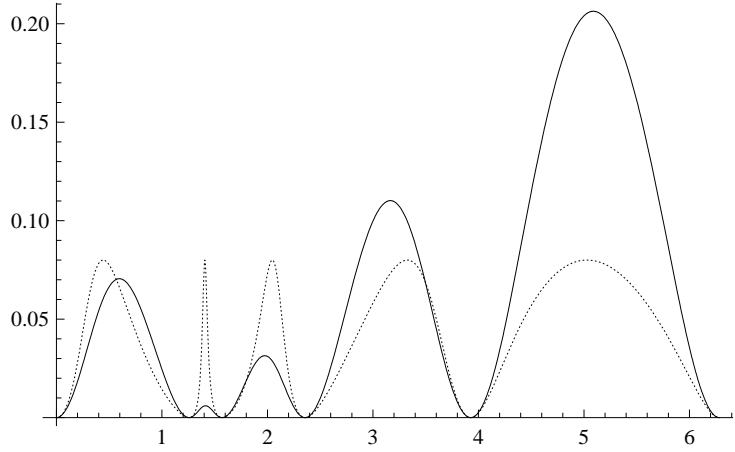}}
  \caption{$1/G_{\mathbf z}(z)$ and its equioscillating approximation (dotted)}
  \label{complex1}
\end{figure}

It is easy to check that the condition of Proposition~\ref{unitzeroprop} on the
position of the zeroes of $g(z)$ is not necessary. Comparing the coefficients
in (\ref{gderpol}) yields that if $n$ is odd, then (\ref{gderpol}) gives a
one-to-one relation between $g(z)$ and $h(z)$, while if $n$ is even, say $n = 2
k$, then regardless of the coefficient of $z^k$ in $g(z)$, we obtain the same
$h(z)$ (in which the coefficient of $z^k$ is zero.) From this, with
Lemma~\ref{circlepol}, we can also see that $g^*(z) = \gamma g(z)$ for some
constant $\gamma$ of modulus 1, and hence the set of zeroes of $g$ must be
symmetric with respect to the unit circle. However, it is unclear that among
such polynomials which ones generate $h(z)$ with zeroes on $T$ only.

 Finally, we note that the above proof of Theorem~\ref{planarpol} suggests an
algorithm that, starting from an initial set $\mathbf z$ of $n$ points on $T$,
transforms it into a new set $\Phi(\mathbf z)$; extremal sets are unaltered,
and numerical experiments suggests that $M(\mathbf{z}) \leq M(\Phi(\mathbf
z))$, therefore $\Phi(\mathbf z)$ is an ``improvement'' of $\mathbf{z}$. The
algorithm goes as follows.

\begin{figure}[h]
\epsfxsize = 9.5cm \centerline{\epsffile{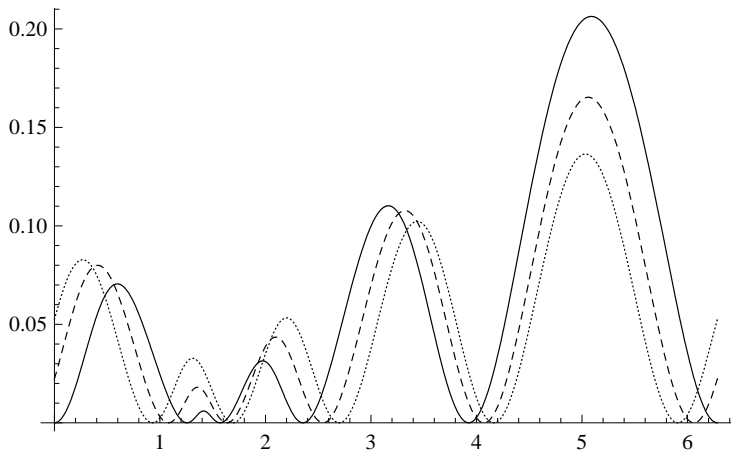}}
  \caption{First three steps of the iteration (plain, dashed, dotted)}
  \label{complex2}
\end{figure}

With the aid of the identities (\ref{sinabs}) and (\ref{abspol}), we obtain the
formula
\[
\sum_{j=1}^{n} \prod_{k \neq j} \sin^2 \left( \frac{t-t_k}{2}\right)
=\frac{(-1)^{n-1}}{4^{n-1} z^{n-1} \prod z_j} \sum_{j=1}^{n}z_j \prod_{k \neq
j}(z-z_k)^2,
\]
where, as usual, $z= e^{i t}$ and $z_j = e^{i t_j}$. Let $\alpha$ be a complex
number satisfying (\ref{alphaeq2}). Since the left hand side is a strictly
positive trigonometric polynomial, using the Fej\'er representation
(\ref{postrig}) (or just by calculating the coefficients of the right hand
side), we obtain that there exists a polynomial $g(z)$ of degree $n$ with roots
in the unit disc only (here 0 is not excluded), such that
\[
-\alpha^2 z\sum_{j=1}^{n} z_j \prod_{k \neq j} (z -z_k)^2 = g(z) g^*(z).
\]
 The polynomial $g(z)$ is unique up to change of
sign, and its leading coefficient is $\pm \, \alpha$.  Let $\Phi(\mathbf z)$ be
the set of the zeroes of the polynomial $g(z) + g^*(z)$; according to
Lemma~\ref{unitzeros}, $\Phi(\mathbf z) \subset T$ . If $\mathbf z$ is an
extremal set, then (\ref{rateq2}) implies that $g(z) = m h(z)$, and therefore
$\mathbf z = \Phi(\mathbf z)$. If, however, the initial set is not extremal,
then we conjecture that the successive iteration of $\Phi$ is converging toward
the (essentially unique) extremal case via sets with larger and larger
$M(\mathbf{z})$. To illustrate  this phenomenon, on Figure~\ref{complex2} we
plot, on the unit circle, the functions $1/G_\mathbf{z}(z)$ generated by the
point sets $\mathbf z$, $\Phi(\mathbf z)$ and $\Phi(\Phi(\mathbf z))$, where
$\mathbf z$ is given by (\ref{points}).

\section{Linear algebraic transformations}\label{linalgsec}

We return to the high dimensional cases of the polarization problems. Clearly,
the complex analytic proof cannot be applied here. However, the
characterisation results that we will obtain, especially
Theorem~\ref{strongpolarextr}, are very similar in spirit to
Lemma~\ref{equimini}: they essentially state that for the locally extremal
vector systems $(u_i)$, there are ``sufficiently many'' points $v$ on the unit
sphere, where $\prod |\la u_i, v \ra|$, or $\sum 1/\la u_i, v \ra^2$, attains
its extremal value.

In the present section, we transform the conjectures to purely linear algebraic
forms. The methods  are closely related to the ones in K. Ball's paper
\cite{ballcomplex}, see also the article of Leung, Li  and Rakesh~\cite{leung}.

Regarding the strong polarization problem, Conjecture~\ref{strongpolar}, it
will suit our purposes better to work with the following equivalent
formulation:
\bigskip

{\em Given a set $u_1, \dots, u_n$ of unit vectors, there is a vector $v$ of
norm $\sqrt{n}$, for which}
\begin{equation}\label{polarnormn}
\sum_{i=1}^{n} \frac{1}{\la u_i, v \ra^2 } \leq n.
\end{equation}

Extremal examples we have seen so far are $n$-dimensional orthonormal systems,
and sets for which  $(\pm u_1, \dots, \pm u_n)$ is equally distributed on the
unit circle. Next, we show that the orthogonal sums of extremal sets are also
extremal, and hence there exist extremal examples of any dimension up to  $n$.
\begin{prop}\label{extremalcases}
Suppose that $(u_1, \dots, u_n) \subset \R^n$ and $(\tilde{u}_1, \dots,
\tilde{u}_m) \subset \R^m$ are extremal with respect to
{\em(\ref{polarnormn})}. Then the system
\[
U= (u_1 \times \tilde{\mathbf{0}}, \dots, u_n \times \tilde{\mathbf{0}},
\mathbf{0} \times \tilde{u}_1, \dots, \mathbf{0} \times \tilde{u}_n),
\]
where $\mathbf{0}$ and $ \tilde{\mathbf{0}}$ are the origin of $\R^n$  and
$\R^m$, is also extremal in $\R^{n+m}$.
\end{prop}

\begin{proof}
Let $\vee = v \times \tilde{v}$ be a point of $\R^{n+m}$ of norm $\sqrt{n+m}$,
where $v \in \R^n$ and $\tilde{v} \in \R^m$. Then $\|v\|^2 +
\|\tilde{v}\|^2=n+m$, and
\[
\sum_{u \in U}\frac{1}{\la u, \vee \ra^2} = \sum_{i=1}^{n} \frac{1}{\la u_i, v
\ra^2 } + \sum_{j=1}^{m} \frac{1}{\la \tilde{u}_i, \tilde{v} \ra^2 } \geq n
\frac{\|v\|^2}{n}+m \frac{\|\tilde{v}\|^2}{m}=n+m.
\]
Equality is achieved by taking $v$ and $\tilde{v}$ to be the points of norm
$\sqrt{n}$ and $\sqrt{m}$ in $\R^n$ and $\R^m$, so as equality holds in
(\ref{polarnormn}).
\end{proof}
We conjecture that all the extremal cases of the strong polarization problem
can be obtained this way, starting from point sets of the unit circle, whose
symmetrized  copies are equally distributed.

The obvious choice of the vector $v$ to satisfy (\ref{polarnormn}) would be the
one which minimises $\sum 1/\la u_i, v \ra ^2$. However, this property does not
lead to conditions on $v$ that are simple to exploit. Instead, we choose a
vector $v$ for which the function $\prod \la u_i , v \ra$ is locally extremal.
Our goal is to show that among these vectors there is one for which
(\ref{polarnormn}) holds. The reason for this approach is that for vectors
which are locally extremal with respect to the product, the following useful
fact holds.

\begin{prop}\label{locmaxrep}
Let $(u_1, \dots, u_n)$ be a system of unit vectors in $\R^d$, and suppose that
for the vector $v \in \R^d$ of norm $\sqrt{n}$, the function
\[
\left| \prod \la u_i, v \ra \right|
\]
is locally maximal. Then
\begin{equation}\label{veq}
v = \sum \frac{u_i}{\la u_i, v \ra}\,.
\end{equation}
\end{prop}

\begin{proof}
The Lagrange multiplier method yields that for a stationary $v$, the gradient
vectors of $\prod \la u_i, v \ra$ and $\|v\|$ are in the same 1-dimensional
subspace: for some $\lambda \in \R$,
\begin{equation*}\label{prodinv}
v = \lambda \sum_{i=1}^{n} \frac{u_i}{\la u_i, v \ra} \; \prod \la u_i, v \ra.
\end{equation*}
Taking inner products of both sides with $v$,
\[
\|v\|^2  = n \, \lambda \prod \la u_i, v \ra.
\]
Since $\|v\|^2 = n$, (\ref{veq}) follows from the previous two equations.
\end{proof}

Defining
\begin{equation} \label{alphav}
\alpha_i = \frac{1}{\la u_i, v \ra}
\end{equation}
and $\alpha  = (\alpha_i)_1^n$, formula (\ref{veq}) transforms to
\begin{equation}\label{ainv}
\left \la u_i, \sum \alpha_j u_j \right \ra = \frac{1}{\alpha_i}\,.
\end{equation}
The following definition is of great importance.

\begin{defi}
The vector $\alpha$ is an {\em inverse eigenvector} of the matrix $M$, if
\begin{equation}\label{inverseigen}
M \alpha = \alpha^{-1},
\end{equation}
where
\[
\alpha^{-1} = \left(\frac{1}{\alpha_1}, \dots, \frac{1}{\alpha_n} \right).
\]
\end{defi}

For two vectors $y, z \in \R^n$, we define their product $yz \in \R^n$ by
$(yz)_i = y_i z_i$. Under this multiplication, $\1$ is the unit element, and
the inverse is given by the above formula.

 The notion of inverse eigenvectors turned up in the solution of the complex
plank problem by K. Ball \cite{ballcomplex}. We will see that they play a
central role in the forthcoming discussion as well. Essentially, both the
complex plank problem and the polarization problems can be formulated as
geometric estimates about the location of inverse eigenvectors, which indicates
that these are very natural objects. As we shall see at the end of the section,
there is a ``duality relation'' between ordinary eigenvectors and inverse
eigenvectors, that is in some sense the same as the duality relation between
the Euclidean ball and the hyperboloid.

If $M$ denotes the Gram matrix of $(u_i)$, that is, $(M)_{ij} = \la u_i ,
u_j\ra$, then the vector $\alpha$ satisfying (\ref{ainv}) is an inverse
eigenvector of $M$. On the other hand, for any such  $\alpha$, the vector $v$
given by
\[
v = \sum \alpha_i u_i
\]
satisfies (\ref{veq}) and (\ref{alphav}), thus $\|v\|^2=n$. Hence the
polarization problem follows from the next statement:

\begin{conj}\label{eigenhyper}
For any real $n \times n$ Gram matrix $M$ with $1$'s on the diagonal, there
exists an inverse eigenvector $\alpha = (\alpha_i)_1^n$, for which
\[
\left| \prod \alpha_i \right| \leq 1.
\]
\end{conj}

The strong polarization problem is implied by the following conjecture:

\begin{conj}\label{eigenball}
For any real $n \times n$ Gram matrix $M$ with $1$'s on the diagonal, there
exists an inverse eigenvector $\alpha = (\alpha_i)_1^n$, for which
\[
\sum \alpha_i^2 \leq n.
\]
\end{conj}

As we have mentioned earlier, the key step in proving the complex plank theorem
is to transform the original problem to the following statement:

\begin{theorem}[\cite{ballcomplex}]\label{complexeigen}
Let $H = (h_{jk})$ be an $n \times n$ complex Gram matrix. Then there are
complex numbers $w_1, \dots, w_n$ of absolute value at most 1, for which
\[
w_j \sum_k h_{jk} \bar w_k = 1
\]
for every $j$.
\end{theorem}

The theorem states that every complex Gram matrix with diagonal $\1$ has an
inverse eigenvector in the complex $l_\infty$ unit ball. Now,
Conjecture~\ref{eigenball}  is the real analogue of Theorem~\ref{complexeigen}
in the sense that the complex $l_\infty$-ball is replaced with the
appropriately scaled  real $l_2$-ball; both of these statements give
fundamental estimates about the location of inverse eigenvectors.

 The geometric difference
between the original polarization problem and the strong version is apparent:
the first essentially asserts that Gram matrices have an inverse eigenvector in
the hyperboloid with boundary
\begin{equation}\label{hypereq}
\hh = \left \{ x = (x_i)_1^n \in \R^n: \left | \prod x_i \right | = 1 \right\},
\end{equation}
while the latter states that there is such a vector even in the inscribed ball
of $\hh$, that is, the standard Euclidean ball of radius $\sqrt{n}$ centred at
the origin.

We shall call real, symmetric, positive semi-definite matrices simply {\em
positive}$\,$; hence, every positive matrix is a Gram matrix (of a system not
necessarily consisting of unit vectors).

Inverse eigenvectors of positive matrices possess a useful geometric property.
Observe that the proof of Proposition~\ref{locmaxrep} and (\ref{alphav}) yields
that if the point $\alpha$  is locally extremal for the function $\prod
\alpha_i$, subject to the condition
\begin{equation}\label{alphacrit}
\left \| \sum \alpha_i u_i \right\| = \sqrt{n},
\end{equation}
which is equivalent to $\alpha^\top A \alpha=n$, then $\alpha$ is an inverse
eigenvector. Proposition~\ref{locmaxrep} would then suggest to look for
minimisers of $|\prod \alpha_i|$. However, the {\em minimum} of the modulus of
the product, subject to the criterium (\ref{alphacrit}), is clearly $0$, and
one rather would like to {\em maximise} it to obtain meaningful information.
The following lemma is a slight modification of the one  in
$\cite{ballcomplex}$.

\begin{lemma}\label{inversefind}
Suppose that $M = (m_{ij})$ is a positive $n \times n$ matrix and that $x=
(x_1, \dots, x_n)$ is a  local maximum point for to the function
\[
\left | \prod x_i \right |
\]
on the $(n-1)$-dimensional manifold defined by
\[
x^\top M x = n.
\]
Then
\[
x_i \sum_{j} m_{ij}\, x_j = 1
\]
for every $i$, that is, $x$ is an inverse eigenvector of $M$.
\end{lemma}

\begin{proof}
By the Lagrange multiplier method, just as in the proof of
Proposition~\ref{locmaxrep}, we immediately obtain that there exists a
$\lambda>0$, for which
\[
x_i \sum_{j} m_{ij} x_j = \lambda
\]
for every $i$. Summing these equations for all $i$, and comparing with $x^\top
M x = n$, yields that $\lambda = 1$.
\end{proof}

For any positive $n \times n$ matrix $M$, the domain
\begin{equation}\label{ellipseq}
\ee = \{x \in \R^n : x^\top M x = n\}
\end{equation}
is a (possibly infinite) $n$-dimensional ellipsoid in the sense that if  $M$ is
singular, say $\rk M = k < n$, then the ellipsoid is obtained by the direct
product of a non-degenerate $k$-dimensional ellipsoid, and $\R^{n-k}$. In this
case, we say that $(n-k)$ axes of the ellipsoid are of infinite length.

The structure of the inverse eigenvalues of $M$ has been described by Leung, Li
and Rakesh in \cite{leung}, see Proposition 3 therein. Let us call the set of
points of $\R^n$ of coordinates with fixed signs, a {\em quadrant of} $\R^n$.
Then $\R^n$ consists of $2^n$ quadrants. It is not complicated to show that
quadrants which intersect $\ker M$ do not contain inverse eigenvectors, while
the others contain exactly one inverse eigenvector. For quadrants $Q$ which do
not intersect $\ker M$, the intersection $Q \cap \ee$ is finite and compact,
and by convexity, it is clear that there is exactly one point $x$ that
maximises $| \prod x_i |$ on $\ee$. By Proposition~\ref{inversefind}, this
point is the unique inverse eigenvector in $Q$.

The ``duality'' between eigenvectors and  inverse eigenvectors in some sense is
equivalent to the relation between the Euclidean ball $B_2^n$ and the
hyperboloid $\hh$, since the eigenvectors are the stationary points on $\ee$
with respect to the Euclidean norm, while the inverse eigenvectors are the
stationary points with respect to the ``product norm'', that is, the modulus of
the product of the coordinates. However, we do not believe that the role played
by inverse eigenvectors is fully understood yet.

\section{The polarization problem}\label{polarsec}

Our goal in this section is to tackle Conjecture~\ref{eigenhyper}, which
implies the original polarization problem, Conjecture~\ref{realpolconj}. We
shall show that the only full-dimensional extremal vector system is the
orthonormal system.

 In view of the previous discussion, using formulas (\ref{ellipseq}) and
(\ref{hypereq}), Conjecture~\ref{eigenhyper} is equivalent to the following
statement:
\bigskip

{\em For any positive matrix $M$ with $1$'s on the diagonal, there is a branch
of the hyperboloid $\hh$ that does not intersect the ellipsoid $\ee$ given by
$x^\top M x = n$. }
\bigskip

The condition $m_{ii}=1$ is equivalent to the fact that $\sqrt{n} \, e_i \in
\ee$ for every $i$, where $(e_i)_1^n$ is the standard orthonormal basis of
$\R^n$.

From now on, $\1$ denotes the vector $(1,\dots, 1)$ in $\R^n$. By scaling, the
statement can be transformed to the following form:

\begin{conj}\label{ellhyp}
Suppose that the matrix $M$ has $\lambda \, \1$ as diagonal. If the ellipsoid
\[
\ee = \{x \in \R^n : x^\top M x = n\}
\]
meets every branch of the hyperboloid $\hh$, then $\lambda \leq 1$.
\end{conj}

Suddenly, we find ourselves in a convenient setting: we can try to characterize
the ellipsoids with maximal diagonal entries among those, that satisfy the
above conditions. Such an approach was first used by Fritz John in his seminal
1948 paper \cite{john}, in which he managed to characterize ellipsoids of
maximal volume inscribed in a convex body. The following formulation of his
result appears in the article of K. Ball \cite{ballintro}.

\begin{theorem}[John \cite{john}]\label{johnell}
Each convex body $K$ contains an unique ellipsoid of maximal volume. This
ellipsoid is $B_2^n$ if and only if $B_2^n \subset K$ and there are vectors
$(u_i)_1^m$ of norm 1 on the boundary of $K$ and positive numbers $(c_i)_1^m$
satisfying
\[
\sum c_i u_i =0
\]
and
\begin{equation}\label{tensorid}
\sum c_i \, u_i \otimes u_i = I_n.
\end{equation}
\end{theorem}
Here $I_n$ denotes the $n \times n$ identity map, and $u \otimes u$ is the rank
one orthogonal projection onto the subspace spanned by $u$. In general, for two
vectors $x, y \in \R^n$, the linear transformation $x \otimes y$ is given by
\[
x \otimes y\, (z) = x\,\la y,z \ra.
\]
Of course, the matrix of this transformation is the tensor product of the two
vectors: if $x=(x_i)_1^n$ and $y = (y_i)_1^n$, then
\[ (x \otimes y )_{ij} = x_i \, y_j.
\]

 The trace of $u \otimes v$ is $\la u, v \ra$. In particular, taking
traces of both sides in (\ref{tensorid}) yields that $\sum c_i =1$, hence the
identity map arises as a {\em convex combination} of the orthogonal
projections.

An equivalent formulation of condition (\ref{tensorid}) is that for each $x \in
\R^n$,
\[
\sum c_i \la x, u_i \ra ^2 = |x|^2.
\]
Hence, the condition essentially says that the contact points between $K$ and
$B_2^n$ behave like an orthonormal basis.

We do not elaborate on the far reaching generalisations of John's theorem that
have been obtained so far; the general method of proving these results is well
described in Ball \cite{ballintro} or, for instance, Bastero and Romance
\cite{bastero}.

Now, for the polarization problem. Our goal is to characterize the ellipsoids
with maximal corresponding diagonal entries among those, that satisfy the
conditions of Conjecture~\ref{ellhyp}. We call $\ee$  {\em locally extremal
with respect to Conjecture}~\ref{ellhyp},  if $\ee$ is given by $x^\top M x
=n$, where $\diag M = \lambda \1$, and for any other ellipsoid in a small
neighbourhood of $\ee$ satisfying the conditions, the diagonal entries are at
most $\lambda$.

\begin{theorem}\label{polarextr}
Suppose that  the ellipsoid $\ee$, given by $x^\top M x = n$, is locally
extremal with respect to Conjecture{\em~\ref{ellhyp}}. If the matrix $M$ is not
singular, then the diagonal of $M$ is at most $\1$, and equality holds only if
\[\ee = \sqrt{n} B_2^n.\]
\end{theorem}

\begin{proof}
Assume that the ellipsoid $\ee$ given by $x^\top M x = n$ meets every branch of
$\hh$, the diagonal of $M$ is of the form $\lambda \, \1$ for some $\lambda$,
and it is locally extremal among such ellipsoids. Let $(u_i)_1^m$ be the set of
contact points between $\ee$ and $\hh$, that is, the collection of the discrete
points of $\hh \cap \ee$. (To visualise, the $u_i$ are contained in the
quadrants where $\ee$ does not ``reach over'' $\hh$.) Note that by
Lemma~\ref{inversefind}, the $u_i$ are inverse eigenvectors of $M$.

The extremality condition yields that for any real, symmetric, $n \times n$
matrix $H$ with $\mathbf 0$ as diagonal and for any positive number $\delta >0$
one of the following is violated:
\begin{itemize}
\item[(a)] the matrix $M + \delta H$ is positive semi-definite;
\item[(b)] $u_i^\top (M + \delta H) \, u_i < n $ for every $i=1, \dots, m$.
\end{itemize}
Since $u_i^\top M u_i = n $ for each $i$, (b) is equivalent to $u_i^\top H u_i
<0 $, for every $i$. Also, if $M + \delta H$ is not positive semi-definite,
then there exists an $x \in B_2^n$, for which $x^\top (M + \delta H)\, x <0$.
This, by compactness, implies, that if for a fixed matrix $H$, the matrix $M +
\delta H$ is not positive semi-definite for any positive $\delta$, then there
exists a point $x \in B_2^n$ for which $x^\top (M + \delta_k H) \, x <0$ for a
sequence $\delta_k \rightarrow 0$. Since $\delta_k x^\top H x \rightarrow 0$,
we necessarily have $x^\top M x =0$.

Therefore, the following holds true:
\bigskip

\noindent {\em There is no real, symmetric, $n \times n$ matrix $H$ with $0$'s
on the diagonal, for which
\[
u_i^\top H u_i <0
\]
for every $i$, and
\[
 x^\top H x \geq 0.
\]
for any vector $x \in \ker M$.  }
\bigskip

This formulation clearly shows the right direction: separation of convex
domains. Introduce the inner product on the space of $n\times n$ real matrices
by
\begin{equation}\label{minner}
\la M, H \ra = \tr (M H) = \sum_{i,j}m_{ij}h_{ij}.
\end{equation}
Sometimes, this is called trace duality. Clearly,
\[
u^\top H u = \la H, u \otimes u \ra.
\]
Note that since $u \otimes u $ is symmetric, we can drop the symmetry condition
on $H$.  Therefore we obtain that for any extremal matrix $M$, the positive
cones
\[
\pos \{ u \otimes u: u \in \R^n \textrm{ is a contact point between } \hh
\textrm{ and } \ee \}
\]
and
\[
\pos \{ x \otimes x: x \in \ker M \}
\]
are not separable with a linear functional of the form $A \rightarrow \la A, H
\ra$, whose kernel contains all the matrices $e_i \otimes e_i$. This implies
that there is a matrix $K$ and a diagonal matrix $D$ such that
\[ K \in \pos \{ x \otimes x: x \in \ker M \}
\]
and
\[
K + D \in \conv \{ u \otimes u: u \in \R^n \textrm{ is a contact point between
} \hh \textrm{ and } \ee \}.
\]
Since for every $x \in \ker M$, we have $\la M, x \otimes x \ra =0$,
\begin{equation}\label{mk}
\la M, K \ra =0.
\end{equation}
On the other hand, let
\begin{equation}\label{kd}
K + D = \sum _{i=1}^m c_i u_i \otimes u_i
\end{equation}
with $\sum c_i =1$, and $c_i \geq 0 $. Then, since $u_i^\top M u_i =n$ for
every $i$,
\[
\la M, K+D \ra = \sum c_i \la M, u_i \otimes u_i \ra = n.
\]
Comparing to (\ref{mk}), we obtain that
\begin{equation}\label{trd}
\la M, D \ra = \lambda \, \tr D = n.
\end{equation}
Now, if $M$ is not singular, then $\ker M = \{ \mathbf 0 \}$, hence $K =
\mathbf 0 \otimes \mathbf 0$. Since $u_i \in \hh$ for every~$i$,
\[
\tr u_i \otimes u_i =  \| u_i \|^2 \geq n,
\]
and therefore, from (\ref{kd}) we obtain that
\[ \tr D \geq n.\]
Comparing this to (\ref{trd}) yields that $\lambda \leq 1$, which is the
desired inequality. If equality holds, then $ \| u_i \|^2 = n $ for every $i$,
which yields that the contact points are vertices of the cube $\{-1, 1 \}^n$.
On the other hand, $\ee$ is compact, and therefore there exists a contact point
in all the quadrants; hence, $\ee$ contains all the points whose coordinates
are $1$ or $-1$. From this, by an inductive argument, it easily follows that
$\ee = \sqrt{n} B_2^n$.
\end{proof}

Theorem~\ref{polarextr} was first proved by Leung, Li and Rakesh \cite{leung}.
If $M$ is degenerate, then the above proof does not go through: by~(\ref{trd}),
we would have to estimate the trace of $D$, and hence, by~(\ref{kd}), it would
be necessary to determine the ``diagonal distance'' of the matrices $u_i
\otimes u_i$ from the positive cone $\pos \{ x \otimes x: x \in \ker M \}$. On
the other hand, this depends on the size of the diagonal of $M$.

It might be possible to rule out the lower dimensional extremal cases by a
different argument, although our efforts in this direction have  not been
successful yet.

In our opinion, the condition $\diag M = \lambda \, \1$ is too strong and too
weak at the same time. It is too strong, because it limits the possible
modifications of the ellipsoid too much. On the other hand, it is too weak to
rule out the lower dimensional extremal cases: if the number of dimensions is
large, then the fixed diagonal provides almost no information about the kernel
of $M$, or, what is the same, the infinite axes of $\ee$. In view of these
facts, a relaxation on the diagonal condition (and hence, posing a stronger
problem) may be fruitful. We shall see in the next section, that the similar
statement derived from the strong polarization problem indeed provides such an
option.

\section{The strong polarization problem}\label{strongpolarsec}

In this section, we transform the strong polarization problem to a geometric
setting, and obtain a characterisation result for the locally extremal systems;
however, this is not explicit enough  to actually determine these systems. We
also give a new proof for the planar case. The failure of the previous proof
for the lower dimensional extremal cases indicate that we have to pay special
attention to these.

Recall that according to Conjecture~\ref{eigenball}, we have to prove that any
$n \times n $ real Gram matrix of a system of unit vectors has an inverse
eigenvector in the ball $\sqrt{n} B_2^n$. In order to handle this problem in a
way that is similar to the previous section, we transform it once more. To this
end, we have to extend the definition of inverse matrices to singular matrices.

\begin{defi}
Let $M$ be a real, symmetric, positive semi-definite $n \times n$ matrix with
eigen-decomposition
\[
M = E D E^\top,
\]
where $D$ is a diagonal matrix, with the eigenvalues $\lambda_1, \dots,
\lambda_n$ of $M$, as diagonal entries.

The {\em generalised inverse} of $M$ is given by
\[
M^{-1} = E D^{-1} E^\top,
\]
where $D^{-1} = \textrm{diag} (1/\lambda_1, \dots, 1/\lambda_n)$ with the
convention that $1/0$ is understood as an abstract symbol $\infty$.
\end{defi}

If $M$ is singular with image space $H$, then $M^{-1}$ maps $H$ onto itself,
and for any $x\in H$, we have $M^{-1}\,M\,x = x$. If $x \notin H$, then the
$\infty$ symbol does not cancel in $M^{-1} x $, and we define  $M^{-1} x $ to
be $\infty$.

If $M$ is not singular, then there is a natural bijection between the inverse
eigenvectors of $M$ and $M^{-1}$: if $\alpha$ is an inverse eigenvector of $M$,
then according to (\ref{inverseigen}),
\[
M \, \alpha  = \alpha^{-1},
\]
and hence,
\[
M^{-1} \alpha^{-1} = \alpha,
\]
therefore $\alpha^{-1}$ is an inverse eigenvector of $M^{-1}$. Clearly, this is
reversible, giving a bijection.

If $M$ is singular, then $M^{-1}$ has no inverse eigenvectors in general.
However, the geometric definition extends to this case as well: Let $\ee$ be
the ellipsoid defined by the equation $x ^\top M x  = n$, and let
\[
\ee^*  = \{ x \in \R^n :  x^\top M^{-1} x  = n \}
\]
be its {\em dual ellipsoid}: the polar with respect to $\sqrt{n} B_2^n$. If $M$
is singular, then $\ee^*$ is lower dimensional: $\dim \ee^* = \rk M$. If $u$ is
a contact point between $\ee$ and $\hh$, then by Lemma~\ref{inversefind},
$u$~is an inverse eigenvector of $M$. It simply follows by an approximation
argument, that regardless of the dimension of $\ee^*$, $u^{-1}$ is a contact
point between $\ee^*$ and $\hh$, and this gives a bijection between contact
points of $\ee$ and $\ee^*$.

Lemma~\ref{inversefind} finds the inverse eigenvectors by maximising $|\prod
x_i|$ on $\ee$.  Now, if $M$ is singular, then $\ee$ is not compact, and the
maximum in question is $\infty$.  Therefore, we rather would like to find the
contact points of $\hh$ and $\ee^*$, as the latter  is always compact.

Assume that $x \in \ee^*$, and $|\prod x_i |$ is locally maximal. Then $x^{-1}$
is an inverse eigenvector of $M$. Recall that for two vectors $y,z \in \R^n$,
their product is taken coordinate-wise. For any $y \in \hh$, we have $|\prod
(yx)_i | = |\prod x_i|$, and thus by the maximality condition, for any $y \in
\hh$ in a sufficiently small neighbourhood of $\1$,
\begin{equation}\label{xym}
(yx)^\top M^{-1} (yx) \geq n.
\end{equation}
Define the matrix $\widetilde M$ by $(\widetilde M)_{ij} = (M)_{ij}/(x_i x_j)$.
It is easy to check that $\widetilde M$ is a positive matrix as well, and its
inverse is given by
\[
(\widetilde M) ^{-1}_{ij} = M^{-1}_{ij} x_i x_j.
\]
Therefore, (\ref{xym}) is equivalent to that for every $y\in \hh$ in a
neighbourhood of $\1$,
\[
y^\top (\widetilde M)^{-1} y \geq n.
\]
If $\diag M = \1$, then
\[
\tr \widetilde M = \sum \frac {1}{x_i^2} = \|x^{-1}\|^2.
\]
Moreover, since $x^{-1}$ is an inverse eigenvector of $M$,
\[
(\widetilde M \, \1)_i = \frac{(M x^{-1})_i}{x_i}  = 1
\]
for every $i$, and hence $\widetilde M \, \1 = \1$.

It would suffice to prove that there exists a contact point of $\ee^*$, for
which the matrix $\widetilde M$ derived in the above way has trace at most $n$.
The natural choice for this contact point is the global maximiser of the
product norm. Then (\ref{xym}) holds for every $y \in \hh$, meaning that the
ellipsoid given by $y^\top (\widetilde M)^{-1} y = n$ is inside $\hh$. Note
that instead of working with the original pair of ellipsoids $\ee$ and $\ee^*$,
we search for a new, different pair.

We shall write $\ee \subset \inte \hh$, if the intersection of $\ee$ and $\hh$
consists of discrete points only. Also, for the sake of simplicity, from now
on, we simply write $M$ for the matrix $\widetilde M$, and $\ee$ and $\ee^*$
for the dual pair of the ellipsoids defined by $y^\top \widetilde M y $ and
$y^\top (\widetilde M)^{-1} y$. The strong polarization problem then follows
from the next statement.

\begin{conj}\label{mconj}
Assume that the positive matrix $M$ satisfies $M \1 = \1$, and that for every
$y \in \hh$, $y ^ \top M^{-1} y \geq n$. Then $\tr M \leq n$.
\end{conj}

The condition $M \1 = \1$ means that $\1$ is a contact point between $\ee$ and
$\hh$. Some condition of this type is clearly necessary, since diagonal
matrices with diagonal of the form  $(a, 1/a, \dots, 1/a)$, where $a$ is a
large positive number,  can have arbitrarily large trace, without $\ee^*$
intersecting $\hh$.

Note, that the contact points between $\ee$ and $\hh$ represent the inverse
eigenvectors of the original Gram matrix: we obtain them by multiplying the
original contact points with the inverse of the original maximiser $x$.

The quantity $\tr M$ has a plainly geometric interpretation: if the axes of
$\ee^*$ are of length $a_1, \dots, a_n$, where $a_i \geq 0$, then
\[
n \, \tr M = \sum a_i^2.
\]

A related result of K. Ball and M. Prodromou \cite{ballprod} states in the
present situation that for any positive matrix $M$ of trace $n$, there is a
vertex of the cube $\{-1,1\}^n$ that is not contained in $\ee$. This, by
duality, implies, that if $\tr M \geq n$, then $\ee^*$ is not contained in the
scaled copy of the $n$-dimensional cross-polytope, which contains the vertices
of $\{-1,1\}^n$ on its boundary.

Next, we give a proof for the planar case of the strong polarization problem,
by proving the above statement in the case when the rank of $M$ is 2.

\begin{proof}[Proof of Conjecture~\ref{mconj} in the case of $\dim \ee^*=2$]
The condition $M \1 = \1$ implies that one axis of $\ee^*$ is $\1$. Let $v$
denote the other axis; then $v \perp \1$, that is,
\begin{equation}\label{vi}
\sum v_i =0.
\end{equation}
 The goal is to show that $\|v \|^2 \geq n(n-1)$. It suffices to
prove, that for any vector $v\ \in \R^n$ with $\|v \|^2 = n(n-1)$, there exists
an angle $\varphi \in [0,2 \pi]$, for which
\[
\left| \prod_{i=1}^n (v_i \sin \varphi + \cos \varphi)  \right| \geq 1.
\]
Let $f_v(\varphi) = \prod (v_i \sin \varphi + \cos \varphi)$. Then $f(\varphi)$
is a trigonometric polynomial of degree at most $n$. Expanding the product and
using (\ref{vi}), we obtain that
\begin{equation}\label{fv}
f_v(\varphi) = \cos^n \varphi + \cos^{n-2}\varphi\, \sin^2 \varphi \sum_{i \neq
j}v_i v_j + \sin^3 \varphi\, Q(\varphi),
\end{equation}
where $g(\varphi)$ is a trigonometric polynomial of degree $\leq n-3$. We
proceed as in the proof of the extremality of the Chebyshev polynomials. Again
by (\ref{vi}),
\[
n(n-1) = \|v \|^2 = \sum_{i \neq j}v_i v_j,
\]
which implies that the first two terms in the expansion of $f_v(\varphi)$ are
independent of $v$. Now, from Section~\ref{remarkscomplexsec} we know, that if
for the original vector system $(u_i)$, $(\pm u_i)$ is equally distributed on
the unit circle, then there are $2n$ contact points between $\ee^*$ and $\hh$.
This can also be checked by a straightforward calculation. Let $v_0$ be the
axis in this extremal case; then $f_{v_0}(\varphi)$ is equioscillating between
$1$ and $-1$ of order $n$ (and it is a multiple of the Chebyshev polynomial).
Assume that for a vector $v$, $\| f_v(\varphi) \|_\infty < 1$. Then
$f_{v_0}(\varphi) - f_v(\varphi)$ has at least $2n$ zeroes, $2n-4$ of which are
different from $0$ or $\pi$. On the other hand, by (\ref{fv}), it can be
written as $\sin^3 \varphi\, h(\varphi)$ for a trigonometric polynomial
$h(\varphi)$ of degree at most $n-3$. Since $h(\varphi)$ can have at most $2n
-6$ zeroes, the difference $f_{v_0}(\varphi) - f_v(\varphi)$ can have at most
$2n-2$ zeroes, and thus $f_{v_0}(\varphi) \equiv f_v(\varphi)$.
\end{proof}

The proof also shows that for the extremal vector systems $(u_i)$, the set
$(\pm u_i)$ is equally distributed on $T$. Also, the ease with which we have
obtained the result compared to Section~\ref{planarcasesec}, well illustrates
the advantage of this form of the strong polarization conjecture over the
original formulation.

As the main result of the section, we prove a characterisation result for the
extremal cases, that is similar to Fritz John's theorem. First, we get rid of
the condition $M \, \1 = \1$. Whenever this holds, $\1$ is an axis of $\ee^*$.
Therefore the ``scaled projection'' $P: \R^n \rightarrow \R^n$ to the
orthogonal subspace of $\1$, defined by
\[
P(x) = \left( x - \1 \frac {\la x, \1 \ra}{n}\right) \frac {1}{\sqrt{1 - (\la
x, \1 \ra/ n)^2}}\, ,
\]
maps $\ee^*$ into $\ee^* \cap \1^{\perp}$. The condition that  $\ee^* \subset
\inte \hh$ is equivalent to $P(\ee^*) \subset \inte P(\hh)$. Clearly, this
projection preserves the duality of $\ee$ and $\ee^*$: the ellipsoids $P(\ee)$
and $P(\ee^*)$ are polar to each other in $\1^\perp$ with respect to $\sqrt{n}
B_2^{n-1}$. Further, the inverse image under $P$ of any symmetric ellipsoid
that is inside $P(\hh)$ is an ellipsoid that satisfies the conditions of
Conjecture~\ref{mconj}.

 Let $P(M)$ be the matrix, for which
\[
P(\ee) = \{ x \in \R^n: x^\top P(M) x = n \};
\]
then
\[
P(M) = M- \frac{ \1 \otimes \1}{n}.
\]
Similarly, $P(M^{-1}) = M^{-1} - \1 \otimes \1 /n$, is the matrix of $\ee^*$.
If $M$ is not singular, then $P(M)$ and $P(M^{-1})$ are inverses of each other
in the sense that
\begin{equation}\label{pminv}
P(M)\, P(M^{-1}) = I_n - \frac{\1 \otimes \1}{n},
\end{equation}
 where the matrix on
the right hand side is the matrix of the projection onto~$\1^\perp$.

The goal is to prove that for any ellipsoid that is inside $P(\hh)$, the sum of
the squares of the axis-lengths is at most $n(n-1)$; that is, the trace of the
matrix of the dual ellipsoid in $\1^\perp$ is at most $n-1$. The advantage is
that at this point, there is no extra assumption on the ellipsoid; on the other
hand, the geometric structure of $P(\hh)$ is more complicated than that of
$\hh$; see the beautiful object on  Figure~\ref{project}.

If $u$ is a contact point between $\ee$ and $\hh$, then $u$ is an inverse
eigenvector of $M$, that is, $M u = u^{-1}$. Now, if $u$ and $v$ are contact
points, then
\[
\la u, v^{-1} \ra = \la u, M v \ra = \la M u, v \ra = \la u^{-1}, v \ra,
\]
because $M$ is symmetric. In particular, choosing $v = \1$, we obtain that for
any contact point $u$,
\begin{equation}\label{u1}
\la u, \1 \ra = \la u^{-1}, \1 \ra.
\end{equation}
If $u$ is a contact point between $\ee$ and $\hh$, then $P(u)$ is a contact
point between $P(\ee)$ and $P(\hh)$, $u^{-1}$ is a contact point between
$\ee^*$ and $\hh$, and $P(u^{-1})$ is a contact point between $P(\ee^*)$ and
$P(\hh)$. A straightforward calculation, using the equation $\la u, u^{-1} \ra
= n$, and (\ref{u1}), reveals that
\[
\la P(u), P(u^{-1}) \ra = n.
\]

\begin{figure}[h]
\epsfxsize =0.7\textwidth \centerline{\epsffile{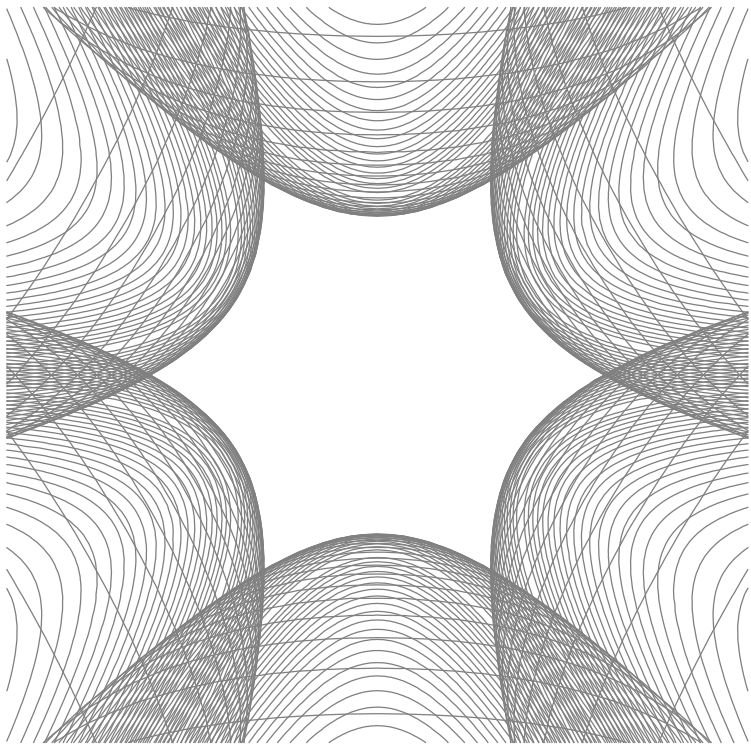}}
  \caption{$P(\hh)$ in the 3-dimensional case}
  \label{project}
\end{figure}

 Local extremality with respect to Conjecture~\ref{mconj} is defined
usually. Locally extremal matrices (and ellipsoids) are characterised via the
following theorem; note that the result holds for lower dimensional extremal
cases as well.

\begin{theorem}\label{strongpolarextr}
Assume that the ellipsoid $\ee^*$, given by $x^\top M^{-1} x = n$, is locally
extremal with respect to Conjecture{\em~\ref{ellhyp}}. Let $(u_i^{-1})_1^m$ be
the set of contact points between $\ee^*$ and $\hh$. Then there exist a series
of positive numbers $(c_i)_1^m$, such that
\begin{equation}\label{pm}
P(M)=\sum_1^m c_i \, P(u_i) \otimes P(u_i^{-1}).
\end{equation}
\end{theorem}

\begin{proof}
If $\ee^*$ is locally extremal, then $P(\ee^*)$ is locally extremal in $P(\hh)$
as well. The set of contact points between $P(\ee^*)$ and $P(\hh)$ are
$(P(u_i^{-1}))$. Moreover, the duality relation implies that the normal
direction to $\hh$ at $u_i$ is parallel to $u_i^{-1}$, and accordingly, the
normal to $P(\hh)$ at $P(u_i)$ is parallel to $P(u_i^{-1})$.

We are going to use a different optimisation argument as in the proof of
Theorem~\ref{polarextr}; the advantage is that this approach automatically
works for lower dimensional extremal ellipsoids as well, therefore, there will
be fewer assumptions on the separating matrix $H$. Moreover, $H$ will not be
required to be symmetric either.

 In order to move the ellipsoid $P(\ee^*)$ so that it stays in
$P(\hh)$, the original contact points cannot cross the local separating
hyperplanes. Since the matrix $P(M)$ is symmetric and positive semi-definite,
it has a symmetric, positive semi-definite square-root $A$, for which $P(\ee^*)
= A \, B_2^n$. Let $H$ be an arbitrary $n \times n$ matrix. We modify the
matrix $A$ to $(I + \delta H)A$, where $\delta
>0$. The image of the contact point $P(u_i^{-1})$ is $P(u_i^{-1}) + \delta H
P(u_i^{-1})$. Then $P(M)$ is modified to
\[
((I + \delta H)A)^2 = P(M) + \delta H P(M) + \delta A H A + \delta^2 HAHA.
\]
The derivative of the above expression with respect to $\delta$ at $\delta = 0$
is $H P(M) + AHA$. Therefore, if the trace increases for small $\delta$, then
\[
\tr H P(M) + \tr A H A = 2 \tr H P(M) > 0.
\]
On the other hand, if the image of $P(\ee^*)$ stays in the bounding box
determined by the separating hyperplanes at the contact points, then for every
$i$,
\[
\la  P(u_i), P(u_i^{-1}) + \delta H P(u_i^{-1})\ra \leq \la P(u_i), P(u_i^{-1})
\ra = n.
\]
Therefore, if $\ee^*$ is locally extremal, then there is no real $n \times n$
matrix $H$, for which
\[
\la P(u_i) , H P(u_i^{-1})\ra \leq 0
\]
for every $i$, and
\[
\tr H P(M)>0.
\]
The only fact we need is that for two vectors $x,y \in \R^n$,
\[
\la H x, y \ra = \la H, x \otimes y \ra,
\]
where the inner product of matrices is defined by  (\ref{minner}). Therefore,
$P(M)$ is not separable from the matrices $P(u_i) \otimes P(u_i^{-1})$ by a
linear functional, which implies that
\[
P(M) \in \pos \{ P(u_i) \otimes P(u_i^{-1})\}\,. \qedhere
\]
\end{proof}

\noindent Note that in the above representation of $P(M)$, the traces of the
matrices are $n$:
\[
\tr P(u_i) \otimes P(u_i^{-1}) = \la P(u_i),P(u_i^{-1})\ra =n.
\]
Hence it would suffice to show that for the coefficients in (\ref{pm}),
\[\sum c_i \leq \frac{n-1}{n}\,.\]

If $M$ is not singular, then we obtain the  result which is analogous to
Theorem~\ref{polarextr}.

\begin{cor}\label{stpolfull}
The only non-degenerate ellipsoid $\ee^*$, that is locally extremal with
respect to Conjecture{\em~\ref{mconj}}, is the unit ball $B_2^n$.
\end{cor}
\begin{proof}
We shall use that for $x, y \in \R^n$ and a real symmetric $n \times n$ matrix
$A$,
\[
(x \otimes y) A = x \otimes (A y).
\]
Assume that $M$ is not singular, and multiply both sides of (\ref{pm}) by
$P(M^{-1})$. Then $P(M^{-1}) P(u_i^{-1}) = P(u_i)$, and by (\ref{pminv}), we
obtain that
\begin{equation}\label{pp}
 I_n - \frac{\1 \otimes \1}{n} = \sum c_i P(u_i) \otimes P(u_i).
\end{equation}
Since $B_2^n \subset \inte \hh$, the norm of every point of $P(\hh)$ is at
least $\sqrt{n}$. Therefore,
\[
\tr P(u_i) \otimes P(u_i) = \la P(u_i),P(u_i) \ra \geq n.
\]
On the other hand,
\[
\tr \bigg( I_n - \frac{\1 \otimes \1}{n} \bigg)= n-1.
\]
Thus, (\ref{pp}) implies that $\sum c_i \leq (n-1)/n$.
\end{proof}

If $M$ is singular, then proceeding the above way, instead of obtaining a
representation of the projection onto $\1^\perp$, we derive a representation of
the projection onto the subspace of $P(\ee^*)$. The problem is that we cannot
estimate the norms of vectors in this subspace; moreover, in the lower
dimensional extremal spaces, the norms of the contact points are not equal.
Some relation connected to the scaling is missing. However, the fact that the
characterisation of Theorem~\ref{strongpolarextr} works for any extremal case,
give us hope that this approach for the polarization problems will eventually
reach its goal.

\chapter{The problem of the longest convex chains}

We discuss the following problem. Let $T$ be the triangle with vertices $(1,0),
(0,0)$ and $(0,1)$. Choose $n$ independent, uniform random points from $T$, the
collection of which is denoted by $X_n$. A subset $Y \subset X_n$ is a convex
chain, if the points are the vertices of a convex path from $(0,1)$ to $(1,0)$.
We are interested in the behaviour of the longest convex chains, where length
is measured by the number of vertices. The maximal length is denoted by $L_n$.
In Section~\ref{expect}, we prove an asymptotic result for the expectation of
$L_n$. It turns out that $L_n$ is highly concentrated about its mean; this is
the main content of Section~\ref{concentr}. With the aid of this property, we
show that the longest convex chains are in a small neighbourhood of a special
parabolic arc with high probability. The proof of this theorem, that is
presented in Section~\ref{limshsec}, is based on a conditional probabilistic
method, to be found in Section~\ref{condsec}. Finally, Section~\ref{experi}
contains numerical results obtained by computer simulations.

Most of these results were published in the joint article with Imre B\'ar\'any
\cite{ab}. However, the material has been almost completely reorganised, and
besides other modifications, the limit shape result is proven by a new method.
Hopefully, these changes led to a more clarified exposition of the topic.

\section{Introduction and related results}\label{introd}

The area of probabilistic discrete geometry has a history that is almost 150
years old. In 1864, Sylvester  posed the following problem in the Educational
Times: ``{\em Show that the chance of four points forming the apices of a
reentrant quadrilateral is 1/4 if they be taken at random in an indefinite
plane}''. It turned out immediately that the problem was wrongly formulated, as
the underlying distribution on the plane was not specified. There are several
ways to correct the question, the most popular of which has been the following:
``{\em Let $K$ be a convex body in the plane, and $n$ points independently from
$K$ with uniform distribution. What is the probability that they are in convex
position?}'' This question was fully answered by B\'ar\'any in 1999 \cite{B99}. The
situation of choosing finitely many, independent, uniform random points in a
convex body has been extremely fertile, and hundreds of papers have been
written on related questions. For an excellent survey of some topics in the
area, see e.g. B\'ar\'any \cite{b08}.

In the present chapter we consider the following problem. Let $T \subset \R^2$
be a triangle with vertices $p_0,p_1,p_2$ and let $X=X_n$ be a random sample of
$n$ independently chosen random points from $T$ with uniform distribution. A
subset $Y \subset X_n$ is a {\sl convex chain} in $T$ (from $p_0$ to $p_2$) if
the convex hull of $Y \cup\{p_0,p_2\}$ is a convex polygon with exactly $|Y|+2$
vertices. A convex chain $Y$ gives rise to the {\em polygonal path} $C(Y)$
which is the boundary of this convex polygon minus the edge between $p_0$ and
$p_2$. The {\sl length} of the convex chain $Y$ is just $|Y|$. We shall
investigate the length of a longest convex chain in $X_n$, which will be
denoted by~$L_n$.

An equally plausible and natural question would be the following. Let $K$ be an
arbitrary convex region in the plane, and choose $n$ random, uniform,
independent points from $K$. What is the expectation of the largest subset of
the random points in convex position? We will explain later in the section,
that using standard methods, this question immediately boils down to the one
above for triangles. We mainly chose to work with the random variable $L_n$
because of the similarity to the famous problem of the longest increasing
subsequences in random permutations.

This connection is easily established. Indeed, let  $X_n$ is a uniform sample
of $n$ independent points from the unit square, and call a chain from $(0,0)$
to $(1,1)$ {\em monotone}, if the slope of every edge in it is positive. By
ordering the points of $X_n$ according to their $x$-coordinates, the order of
their $y$-coordinates represents a permutation $\sigma$ of $\{ 1, \dots , n
\}$, where the longest monotone chain from $X_n$ corresponds to the longest
increasing subsequence of~$\sigma$. It is easy to check that the resulted
distribution on $S_n$ is the uniform distribution.

The problem of the longest increasing subsequences is over 40 years old, in
which period, it has been linked to various parts of mathematics, e.g. Young
tableaux, patience sorting, planar point processes, and the theory of random
matrices. In 1977, independently of each other, Logan and Shepp \cite{ls} and
Vershik and Kerov \cite{vk} proved that the expectation is $2\sqrt n(1+o(1))$.
After a variety of  related results, the limit distribution was determined by
Baik, Deift and Johansson in 1999 \cite{baik}. Surprisingly, this turned out to
be the same as the limit distribution of the largest eigenvalue of random $n
\times n$ Hermitian matrices, which was determined by Tracy and Widom in 1994
\cite{tracy}. More details about the problem can be found in the excellent
survey paper by Aldous and Diaconis \cite{ad}.

Let us return now to random points in a planar convex body. We need some
notation to formulate the results. The (Lebesgue) area of $K$ will be denoted
by $A(K)$. Next, we define the {\em affine perimeter} of a planar convex body
$S$, see e.g. \cite{B97} or \cite{BP}; there are many ways to do it, of which
we choose the one which is most relevant. Choose a partition $x_1, \dots, x_m,
x_{m+1}=x_1$ of the boundary $\partial S$, and for every $i$, let the line
$l_i$ support $S$ at $x_i$. Denote by $T_i$ the triangle enclosed by $l_i$,
$l_{i+1}$, and the segment $x_1 \, x_{i+1}$. The affine perimeter of $S$ is
then defined by
\begin{equation}\label{affper}
AP(S) = 2 \lim \sum_{1}^m \sqrt[3]{A(T_i)},
\end{equation}
where the limit is taken over a sequence of partitions whose mesh tends to 0.
The existence of the limit is showed in the above cited papers. We also mention
that if the boundary of $S$ is twice differentiable, then $AP(S) =
\int_{\partial S}\kappa^{1/3} ds$, where  $\kappa$ denotes the curvature of
$\partial S$ and the integral is taken with respect to the arc length.
Naturally, the affine length can be defined in the same manner for convex
curves. Finally, for a convex body $K \subset \R^2$, let
\begin{equation}\label{maxaffper}
A^*(K) = \sup \{AP(S): S \subset K \textrm{ convex} \}.
\end{equation}
Because of compactness, the supremum in the above definition is attained.
Furthermore, B\'ar\'any proved in \cite{B97} the maximiser is unique: there is
exactly one convex body contained in $K$, say $K_0$, for which $A^*(K) =
AP(K_0)$. Now, the structure of $K_0$ is well known as well \cite{B97}. If $T$
is the triangle with vertices $p_0,p_1,p_2$, then among all convex curves
connecting $p_0$ and $p_2$ within $T$, the unique parabola arc $\G \subset T$
that is tangent to the sides $p_0p_1$ at $p_0$ and $p_1p_2$ at $p_2$ has the
largest affine length. The parabola arc $\G$ will be called {\em the special
parabola} in $T$, see Figure~\ref{parab1}. From this fact, it simply follows
that $\partial K_0 \setminus \partial K$ consists of parabola arcs that are
tangent to $\partial K$ at their endpoints.

\begin{figure}[h]
\epsfxsize = 0.4 \textwidth \centerline{\epsffile{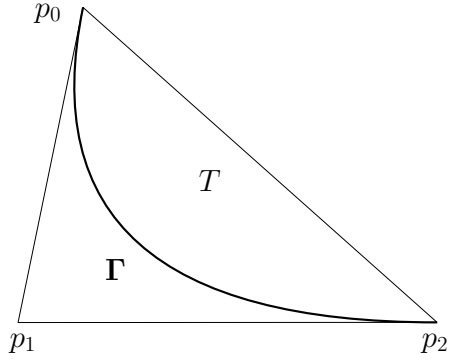}}
  \caption{The special parabola}
  \label{parab1}
\end{figure}

The importance of affine perimeter was first pointed out in the work of R\'enyi
and Sulanke \cite{RS}. They proved that if $K$ is a smooth convex body in the
plane, and $X_n$ is a uniform sample of $n$ points in $K$, then the expected
number of the vertices of $\conv X_n$ (the convex hull of $X_n$) is
asymptotically
\begin{equation}\label{rsexp}
\Gamma\left(\frac{5}{3}\right) \sqrt[3]{\frac{2}{3}} \left(A(K) \right)^{-1/3}
AP(K) \, \sqrt[3]{n} .
\end{equation}
Here of course $\Gamma$ stands for the Gamma function.

Now, let $K\subset \R^2$ be an arbitrary convex body. B\'ar\'any \cite{B99} proved
that if we take all the convex polygons spanned by points of $X_n$, then the
majority of them is close to $K_0$. Therefore, $K_0$ is the {\em limit shape}
of the random convex polygons inside $K$. In the article \cite{BRSZ},
strengthening this result, an almost sure limit theorem and a central limit
theorem for convex chains are proved.

Next, we give an alternative model for obtaining a sample from $K$, by choosing
{\em lattice points}. The connection between this and the uniform model
motivates a strong inspiration on both sides. Consider the lattice $(1/t)\,
\Z^2$, where $\Z^2$ is the usual integer lattice in $\R^2$ and $t>0$ is large,
and set $X=K \cap (1/t)\,\Z^2$.  The same questions can be formulated for $X$
as for $X_n$. For instance, Andrews \cite{andr} proved an upper estimate for
the number of vertices of the integer convex hull:
\[
|\conv X| \leq c\, A(K)^{1/3}
\]
for some constant $c$. This shows that the two models do not behave necessarily
in the same way. However, for the convex lattice polytopes of $X$, the same
limit shape result holds as for the uniform model. This, in full generality,
was proved in \cite{B97}.

Regarding the present problem, there exists a result about {\em convex lattice
chains}. Let $T$ be the usual triangle, and consider the lattice points in $T$.
Let $n= |X|$; clearly, for large $t$, $n \approx \A(T)\,t^2$. Write $Y_n\subset
X$ for a longest convex chain in $T$. It is shown by B\'ar\'any and Prodromou
\cite{BP} that, as~$t \to \infty$ (or $n\to \infty$),
\begin{equation}\label{tria}
|Y_n|=\frac 6{(2\pi)^{2/3}} \sqrt[3]{t^2 \A(T)}(1+o(1))=\frac
6{(2\pi)^{2/3}}n^{1/3}(1+o(1)).
\end{equation}

Next, we list our results. First, in Section~\ref{expect}, we prove an
asymptotic result about the expectation of $L_n$. In view of (\ref{tria}), we
expect the order of magnitude to be $n^{1/3}$. This is indeed so:
\begin{theorem} \label{limit}
There exists a positive constant $\al$, for which
$$\lim_{n \rightarrow \infty} \frac{ \E L_n}{\sqrt[3]{n}}=\al \,.$$
\end{theorem}
We also establish the estimates $1.57 < \alpha < 3.43$. Numerical experiments
suggest that $\al=3$ and we venture to conjecture that this is the actual value
of $\al$, which would nicely match the expectation of the longest increasing
subsequences.

We have seen that the convex chains are located close to the parabolic arc $\G$
in both the uniform and the lattice models. Although this does not imply that
the {\em longest} convex chains are close to $\G$, it is reasonable to guess
so. Indeed, we will essentially prove that if $\cc(X_n)$ is the collection of
all longest convex chains from $X_n$, then for every $\eps>0$,
\begin{equation}\label{limelso}
\lim _{n \to \infty} \PP \big(\dist(C(Y),\G)> \eps\mbox{ for some }Y \in
\cc(X_n)\,\big)=0,
\end{equation}
where $\dist(.,.)$ stands for the Hausdorff distance. In order to achieve this
result, we need strong concentration results in the sense of Talagrand's
inequality, that are established in Section~\ref{concentr}.
Sections~\ref{condsec} and \ref{limshsec} contain the proof of the following
quantitative limit shape theorem.
\begin{theorem}\label{limsh}
Let $\gamma \geq 1 $ and define $\eps= 2 \gamma^{1/2} \, n^{-1/12} (\log
n)^{1/4}$. Then there exists $N>0$, depending on $\gamma$, such that for every
$n>N$,
\[
\PP\big( {\rm dist}(C(Y),\G)>\eps \mbox{ for some }Y \in \cc(X_n)\big)<
n^{-\gamma^2/14}.
\]
\end{theorem}
\noindent Note that this is much stronger than (\ref{limelso}), since here
$\eps \rightarrow 0$ as well.

Once Theorems~\ref{limit} and \ref{limsh} for triangles are known, it is not
too hard to extend the results for arbitrary convex sets: the method is
illustrated for example in \cite{BP}. We give a sketch here. Let $K\subset
\R^2$ be a convex body, and $K_0$ its convex subset of maximal affine
perimeter. Let $X_n$ be a uniform sample from $K$, and assume that $Y_n$ is a
largest subset of $X_n$ in convex position. Let $C(Y_n)$ be the polygonal path
of $Y$. Let $m$ be fixed, and for every $k=1, \dots, m$, let $l_k$ be a tangent
to $C(Y_n)$ of direction $2 \pi k/m$, with contact point $x_k$. Let $T_k$ be
the triangle delimited by $l_k$, $l_{k+1}$, and the segment $x_k \, x_{k+1}$.
It is easy to check that $Y_n \cap T_k$ is a maximal convex chain in $T_k$. For
$n \gg m$, the number of points in $T_k$ is about $n A(T_k)/A(K)$, and by
Theorem~\ref{limit} and formulae (\ref{affper}) and (\ref{maxaffper}),
\[
|Y_n| \approx \frac{\alpha\, n^{1/3}}{\sqrt[3]{A(K)}} \sum_1^m \sqrt[3]{A(T_k)}
\leq \frac{\alpha\, n^{1/3} A^*(K)}{2 \sqrt[3]{A(K)}},
\]
where $\al$ is the constant from Theorem~\ref{limit}. On the other hand, by
choosing the points $x_1, \dots, x_m$ on $\partial K_0$, the quantity on the
right hand side can be achieved. Therefore,
\[ \lim_{n \rightarrow \infty} n^{-1/3} \E|Y_n|= \frac
{\al \A^*(K)}{2\root 3 \of{A(K)}}.
\]
Moreover, this is achieved only if $C(Y)$ is sufficiently close to $K_0$;
therefore, the limit shape of the maximal convex polygons is necessarily $K_0$.

The above argument serves as the basis for the subsequent proofs as well.
However, there are many non-trivial details, and tricky proofs, to come.

Finally, we mention another model, that is often used in probabilistic
geometry, although in the present case it will not be applied. Let $X(n)$ be a
homogeneous planar Poisson process of intensity $n/ \A(T)$. Given a domain $D$
in the plane, the number of points of $X(n)$ in $D$, that we denote by $m(D)$,
has Poisson distribution with parameter $\lambda=n \A(D)/\A(T)$:
\[
 \PP\big(m(D)=k\big)= e^{-\lambda} \lambda ^k / k! \;.
\]
We can also think of the Poisson model as follows: for a domain $D$, we first
pick a random number $m$ according to the corresponding Poisson distribution,
and then choose $m$ random, independent, uniform points in $D$. As is well
known, the uniform model $X_n$ and the Poisson model $X(n)$ behave very
similarly. In particular, Theorems \ref{limit}, \ref{limsh}, and
 \ref{conc1} remain valid for the latter as well, with essentially the
same quantitative estimates. Obtaining these results from the ones presented
here is straightforward. An equally standard way is to prove the theorems for
the Poisson model, and then transfer the results to the uniform model. The
disadvantage of obtaining slightly weaker results is balanced by the fact that
under the Poisson model, the number of points of $X(n)$ in two disjoint domains
are independent random variables.

The longest increasing subsequence problem has been almost completely solved by
now, see \cite{ad}. In this respect, our results only constitute the first, and
perhaps the simplest steps in understanding the random variable $L_n$.

\section{Geometric and probabilistic tools}\label{geotools}

When choosing a random point in the triangle $T$, the underlying probability
measure is the normalized Lebesgue measure on $T$. Most of the random variables
treated in this chapter (e.g. $L_n$) are defined on the $n$th power of this
probability space, to be denoted by $T^n$. In this case $\PP$ denotes the $n$th
power of the normalized Lebesgue measure on $T$. Plainly, choosing $n$
independent random points in $T$, the number of points in any domain $D \subset
T$ is a binomial random variable of distribution $B(n,\A(D)/\A(T))$, and the
expected number of points in $D$ is $n \A(D)/\A(T)$.

For binomial random variables we have the following useful deviation estimates,
which are relatives of Chernoff's inequality, see \cite{as}, Theorems A.1.12
and A.1.13. If $K$ has binomial distribution with mean value $k>1$ and $c>0$,
then
\begin{equation} \label{hoeffding}
\PP \big(K \leq k -c \sqrt{k \log k}\,\big) \leq k^{-c^2/2}.
\end{equation}
On the other hand, for $c>1$,
\begin{equation} \label{hoeffdingfelso}
\PP \big(K \geq ck \big) \leq \left( \frac{e}{c}\right)^{ck}.
\end{equation}
We will use (\ref{hoeffding}) often, mainly with $c=1$.

\begin{figure}[h]
\epsfxsize =\textwidth \centerline{\epsffile{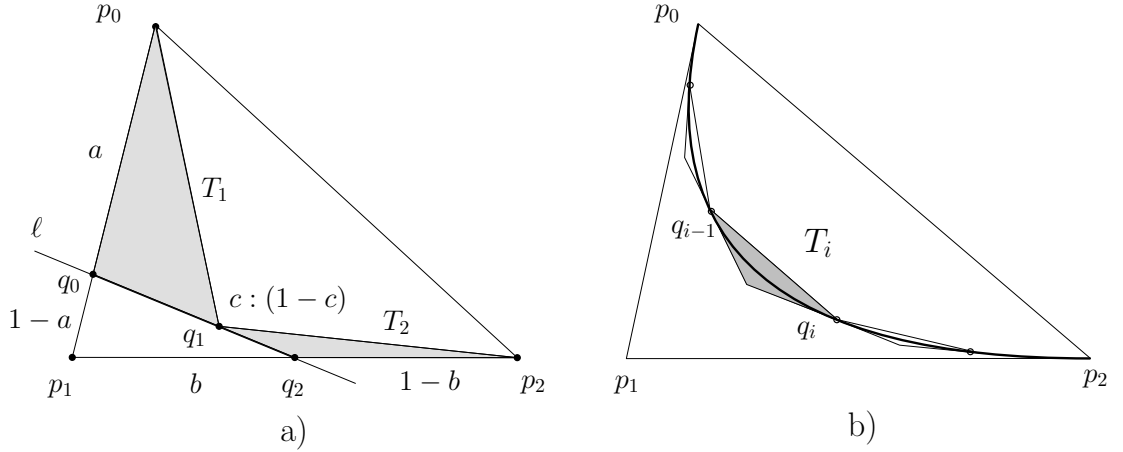}}
  \caption{Characterisation of $\G$}
  \label{mobiusfig}
\end{figure}

As we have mentioned earlier, the special parabola arc $\G \subset T$ is
characterized by the fact that it has the largest affine length among all
convex curves connecting $p_0$ and $p_2$ within $T$.  This is a consequence of
the following theorem from \cite{Bla}. Assume that a line $\ell$ intersects the
sides $p_0 p_1$ resp. $p_1 p_2 $ at points $q_0$ and $q_2$. Let $q$ be a point
on the segment $ q_0 q_2$ and write $T_1$ resp. $T_2$ for the triangle with
vertices $p_0,q_0,q$ resp. $q,q_2,p_2$, see Figure~\ref{mobiusfig} a).

\begin{theorem}[\cite{Bla}]\label{mobius} Under the above assumptions
\[
\sqrt[3]{\A(T_1)}+\sqrt[3]{\A(T_2)} \leq \sqrt[3]{\A(T)}.
\]
Equality holds here if and only if $q_1 \in \Gamma$ and $\ell$ is
tangent to $\G$ at $q_1$.
\end{theorem}

The equality part of the theorem implies the following fact. Assume that
$p_0=q_0,q_1,\dots,q_k=p_2$ are points, in this order, on $\G$. Let $T_i$ be
the triangle delimited by the tangents to $\G$ at $q_{i-1}$ and $q_i$, and by
the segment $q_{i-1} q_i$, $i=1,\dots,k$; see Figure~\ref{mobiusfig} b).

\begin{cor}\label{areasum} Under the previous assumptions
$\sum_{i=1}^k\sqrt[3]{\A(T_i)}= \sqrt[3]{\A(T)}$. In particular,
when $\A(T_i)=t$ for each $i=1,\dots,k-1$ and $\A(T_k) < t$, then
$k-1\le \sqrt[3]{\A(T)/t}<k$.
\end{cor}

We will need a strengthening of Theorem \ref{mobius}. Assume $q_0$ resp. $q_2$
divides the segment $p_0 p_1$ resp. $p_1 p_2$ in ratio $a:(1-a)$ and $b:(1-b)$,
see Figure~\ref{mobiusfig} a).

\begin{theorem}\label{mobi+} With the above notation
\[
 \sqrt[3]{\A(T_1)}+\sqrt[3]{\A(T_2)} \leq \sqrt[3]{\A(T)}-
\sqrt[3]{\A(T)}\frac 13 (a-b)^2.
\]
\end{theorem}

\begin{proof} Let $c$ be a number between $0$ and $1$,
so that $q_1$ divides the segment $q_0 q_2$ in ratio $c:(1-c)$. Then, writing
$\A(xyz)$ for the area of the triangle with vertices $x,y,z$,
\[
\A(p_0q_0q_1) = a \A(p_0p_1q_1) = ac \A(p_0p_1q_2) = abc \A(p_0p_1p_2),
\]
showing $\A(T_1)=abc\A(T)$. Similarly, $\A(T_2) = (1-a)(1-b)(1-c) \A(T).$ Hence
we have to prove the following fact: $0 \leq a,b,c \leq 1$ implies
\begin{equation} \label{abc}
1 - \sqrt[3]{abc} - \sqrt[3]{(1-a)(1-b)(1-c)} \geq \frac{1}{3}
(a-b)^2.
\end{equation}
Denote $Q$ the left hand side of (\ref{abc}). By computing the
derivative of $Q$ with respect to $c$ yields that for fixed $a$ and
$b$, $Q$ is minimal when
\[
c=\frac{\sqrt{ab}}{\sqrt{ab}+\sqrt{(1-a)(1-b)}} \, .
\]
It is easy to see that with this $c$,
\[
\sqrt[3]{abc} + \sqrt[3]{(1-a)(1-b)(1-c)} =
\left(\sqrt{ab}+\sqrt{(1-a)(1-b)}\right)^{2/3}.
\]
Now, denote $\left(\sqrt{ab}+\sqrt{(1-a)(1-b)}\right)^2$ by $1-u$, that is,
\[
u = a+b-2ab-2\sqrt{ab(1-a)(1-b)}.
\]
We claim that $u \geq (a-b)^2$: this is the same as
\[
a-a^2 + b- b^2 \geq 2\sqrt{(a-a^2)(b-b^2)},
\]
which is just the inequality between the arithmetic and geometric means for the
numbers $a-a^2,b-b^2\ge 0$. Therefore, using $u \leq 1$,
\[
Q \geq 1-(1-u)^{1/3} \geq \frac {1}{3}\,u \geq \frac{1}{3}\,(a-b)^2.
\qedhere
\]
\end{proof}

Theorems \ref{mobius} and \ref{mobi+} imply the following statement.

\begin{cor}\label{erinto}
If $q \in \G$ and $\ell$ is tangent to $\G$ at $q$, then with the above
notations, $a=b$.
\end{cor}

Since an affine transformation does not influence the value of $L_n$, the
underlying triangle $T$ can be chosen arbitrarily. Our standard model for $T$
is the one with $p_0=(0,1)$, $p_1=(0,0)$, $p_2=(1,0)$ as the vertices of $T$.
In this case the special parabola $\G$ is given by the equation $\sqrt{x} +
\sqrt{y}=1$.

Finally, we will give strong concentration results for $L_n$ with the help of
Talagrand's following inequality \cite{tal}. Suppose that $Y$ is a real-valued
random variable on a product probability space $\Omega^{\otimes n}$, and that
$Y$ is 1-Lipschitz with respect to the Hamming distance, meaning that
\[
|Y(x)-Y(y)| \leq 1
\]
whenever $x$ and $y$ differ in one coordinates. Moreover assume that $Y$ is
{\sl $f$-certifiable}. This means that there exists a function $f: \mathbb{N}
\rightarrow \mathbb{N}$ with the following property: for every $x$ and $b$ with
$Y(x) \geq b$ there exists an index set $I$ of at most $f(b)$ elements, such
that $Y(y) \geq b$ holds for every $y$ agreeing with $x$ on $I$. Let $m$ denote
the median of $Y$. Then for every $s>0$ we have
\begin{equation}\label{tala}
\begin{aligned}
&\PP (Y \leq m-s)\leq 2 \, \mathrm{exp}\left(\frac{-s^2}{4f(m) } \right),\\
&\PP (Y \geq m+s)\leq 2 \, \mathrm{exp}\left(\frac{-s^2}{4f(m+s) } \right).
\end{aligned}
\end{equation}

\section{Expectation}\label{expect}

The aim of this section is to prove of Theorem~\ref{limit}. We also establish
upper and lower bounds for the constant $\alpha$.

\begin{proof}[Proof of Theorem~\ref{limit}]
We start with an upper bound on $\E L_n$:
\begin{equation}\label{felso}
\limsup_{n \rightarrow \infty} \frac{\E L_n}{\sqrt[3]{n}} \leq
\sqrt[3]{2}e = 3.4248\dots.
\end{equation}
It is shown in \cite{B99}, equation (5.3) (cf. \cite{BRSZ} as well)
that the probability of $k$ uniform independent random points in $T$
forming a convex chain is
\[
\frac{2^k}{k! \,(k+1)!} \;.
\]
Therefore we can estimate  the probability that a convex chain of length $k$
exists:
\begin{equation*}\label{upp}
\PP(L_n\ge k) \le {n \choose k} \frac{2^k}{k! \,(k+1)!}\;.
\end{equation*}
We use this estimate and Stirling's formula to bound $\E L_n$.
Assume $\gamma > \sqrt[3]{2}e$. Then
\begin{align*}
\E L_n &= \sum_{k=0}^{n} \PP(L_n>k) \leq \sum_{k=0}^{n} \PP(L_n \ge k)\\
&\leq \gamma \sqrt[3]{n} + \sum_{k > \gamma \sqrt[3]{n}}
\PP(L_n\ge k) \\
&\leq  \gamma \sqrt[3]{n} + \sum_{k > \gamma \sqrt[3]{n}} {n \choose
k} \frac{2^k}{k! \,(k+1)!}\\
&\leq \gamma \sqrt[3]{n} + \sum_{k > \gamma \sqrt[3]{n}} \frac{(2n)^k}{(k!)^3} \\
&\leq \gamma \sqrt[3]{n} + \sum_{k > \gamma \sqrt[3]{n}}
\frac{1}{\sqrt{(2\pi \gamma)^3 n}} \left(\frac{2\,e^3}{\gamma^3}\right)^k\\
&\leq \gamma \sqrt[3]{n} + n^{-1/2} C,
\end{align*}
where $C = \gamma^3 / (\gamma^3 - 2 e^3)$ is a positive constant.
Since this holds for arbitrary $\gamma > \sqrt[3]{2}\,e$,
(\ref{felso}) is proved.

Next, we give a lower bound for $\E L_n$. We apply Corollary~\ref{areasum} with
$t=2\A(T)/n$, obtaining a set of triangles $\{T_1, \dots, T_k\}$, where for $1
\leq i \leq k-1$, the area of $T_i$ is $t$, and the area of $T_k$ is less than
$t$. By (\ref{areasum}), $ k \geq \sqrt[3]{n/2}$. Let $X_n$ be the uniform
independent sample from $T$. Let $q_i$ be a point of $T_i \cap X_n$, provided
that $T_i \cap X_n \neq \emptyset$. The collection of such $q_i$'s forms a
convex chain. Hence, the expected length of the longest convex chain is at
least the expected number of non-empty triangles $T_i$:
\begin{align*}
\E L_n \; &\geq \; \sum_1^k \PP\big(T_i\cap X_n \ne \emptyset\big)
\; \geq \;
(k-1)\left(1-\left(1-\frac{2}{n}\right)^n\right) \\
&\ge \left(\sqrt[3]{\frac{n}{2}}-1\right) \, \left(1-e^{-2}\right)
\approx 0.6862 \, n^{1/3}.
\end{align*}

\noindent What we have proved so far is that
\[
\underline{\al} = \liminf_{n \rightarrow \infty}n^{-1/3}\E L_n >
0.6862,\mbox{ and } \overline{\al}= \limsup_{n \rightarrow \infty}
n^{-1/3}\E L_n < 3.4249.
\]
We show next that the limit exists. Suppose on the contrary that
$\underline{\al}< \overline{\al}$.

The idea of the proof is to use Corollary~\ref{areasum} again with parameters
chosen so that the expected length of the longest convex chains in the small
triangles is close to $\overline{\al}$, while for the triangle $T$, $\E L_n$ is
close to $\underline{\al}$. This will result in a contradiction.

Choose a large $n$ with $ \E L_n\geq (1-\eps)\,\overline{\al} \sqrt[3]{n}$, and
an $N$ much larger than $n$ with $ \E L_N\le (1+\eps)\, \underline{\al}
\sqrt[3]{N}$. Here $\eps$ is a suitably small positive number. Define $n_1$ by
the equation $n = n_1 - \sqrt{n_1 \log n_1}$.

Choose $N$ uniform, independent random points from triangle $T$. Define $t =
n_1 / N$. Hence the expected number of points in a triangle of area $t$
(contained in $T$) is $n_1$.

Applying  Corollary \ref{areasum} with this $t$ yields the set of triangles
$T_1, \dots, T_k$, where  $k > \sqrt[3]{N/ n_1}$.

Denote by $k_i$ the number of points in $T_i$, and by $\E L^i$ the expectation
of the length of the longest convex chain in $T_i$. Clearly, $k_i$ has binomial
distribution with mean $n_1$, except for the last triangle where the mean is
less than $n_1$.

Since the union of convex chains in the triangles $T_i$ is a convex chain in
$T$ between $(0,0)$ and $(1,1)$, by the estimate (\ref{hoeffding}) we have
\begin{align*}
 \E L_N &\geq \sum_{i \leq k}  \E L^i \; \ge \;
 \sum_{i\leq k-1} \PP (k_i > n)  \E L_{n}\\
&\geq \sum_{i\leq k-1} \left(1- n_1^{-1/2}\right) (1-\eps) \,
\overline{\al}
\sqrt[3]{n}\\
&\geq \left(\sqrt[3]{N/ n_1}-1\right)\left(1- n_1^{-1/2}\right)
(1-\eps) \,\overline{\al} \sqrt[3]{n}\\
&=\overline{\al}\,
\sqrt[3]{N}(1-\eps)\left(1-n_1^{-1/2}\right)\left(\sqrt[3]{n/n_1}-\sqrt[3]{n/N}\right)\\
&\ge \overline{\al} \,\sqrt[3]{N}(1-2\eps),
\end{align*}
where the last inequality holds if $n$ is chosen large enough and
$N$ is chosen even larger with $n/N$ very small. Thus $(1+\eps)\,
\underline{\al} \ge (1-2\eps)\, \overline{\al}$ which, for small
enough $\eps$, contradicts our assumption $\underline{\al} <
\overline{\al}$.
\end{proof}

 The lower bound $\E L_n \ge 0.6862 \, n^{1/3}$ is
probably the easiest to prove. A better estimate, also mentioned by Enriquez
\cite{En}, can be established by the following sketch. Assume $T$ is the
standard triangle and let $D$ denote the domain of $T$ lying above $\Gamma$.
Then $\A(D) = 1/3$, so the expected number of points in $D$ is $2n/3$, and the
number of points is concentrated around this expectation. The affine perimeter
of $D$ is $2 \sqrt[3]{1/2}$ (see \cite{B99}), and (\ref{rsexp}) yields that the
expected number of vertices of $\conv(D \cap X_n)$ is about
\[ \Gamma\left(\frac{5}{3}\right) \sqrt[3]{\frac{2}{3}}
\left(\frac{1}{3} \right)^{-1/3} 2\sqrt[3]{1/2} \, \sqrt[3]{2n/3} \approx
1.5772 \,\sqrt[3]{n}.\]
Since most vertices are located next to the parabola,
the majority of them form a convex chain, and so
\begin{equation}\label{also}
\liminf_{n \rightarrow \infty} \frac{ \E L_n}{\sqrt[3]{n}} \geq
1.5772\dots.
\end{equation}
This estimate leads to slightly stronger quantitative results, and thus from
now on, we will use it instead of $\alpha > 0.6862$.

\newpage

\section{Concentration results}\label{concentr}

In this section, we prove strong concentration results for $L_n$ and related
variables. We will use Talagrand's inequality (\ref{tala}), see
Section~\ref{geotools}. When applied to $L_n$, this yields a concentration
result about the median, what we denote by $m_n$. However, we want to prove
that $L_n$ is close to its expectation. Luckily, concentration ensures that the
mean and the median are not far apart; in fact, it will turn out that $\lim
n^{-1/3}m_n = \al$.

First, we need a lower bound on $m_n$.
\begin{lemma}\label{median}  Suppose that $\log n > 25$.
Then $$m_n \ge \sqrt[3]{3n/\log n}.$$
\end{lemma}
\noindent Since this is a special case of Lemma \ref{mediansub}, the proof will
be given there.

Here comes our first, basic concentration result for $L_n$.

\begin{theorem}\label{conc1}
For every $\gamma>0$ there exists a constant $N$, such that for every $n>N$,
$$\PP \big(|L_n -  \E L_n| > \gamma \sqrt{\log n} \; n^{1/6}\big) <
n^{-\gamma^2/14}.  $$
\end{theorem}

\begin{proof} The statement cries out for
the application of Talagrand's inequality. The random variable $L_n$ satisfies
the conditions with $f(b)=b$, since fixing the coordinates of a maximal chain
guarantees that the length will not decrease, and changing one of the points
changes the length of the maximal chain by at most one. Write $m=m_n$ for the
median in the present proof. Setting $s=\be\sqrt{m \log m}$ where $\be$ is an
arbitrary positive constant, (\ref{tala}) implies that
\begin{align*}
\PP \big(|L_n-m| \geq \be \sqrt{m \log m}\,\big) &< 4
\exp\left\{\frac
{-\be^2 m\log m}{4(m+\be \sqrt{m\log m})} \right\}\\
&= 4 \exp\bigg\{\frac{-\be^2\log m}{4(1+\be \sqrt{m^{-1}\log m})} \bigg\}
\end{align*}

\noindent Define now $\be_0= c \sqrt{m / \log m}$ with a constant $c>0$, which
will be specified at the end of the proof in order to give the correct
estimate. If $\be \leq \be_0$, then $\be\sqrt{m^{-1}\log m} \le c$, and the
denominator in the exponent is at most $4(1+c)$. Thus
\begin{equation}\label{betasmall}
\PP \big(|L_n-m| \geq  \be \sqrt{m \log m}\, \big) <
4m^\frac{-\be^2}{4(1+c)}.
\end{equation}
On the other hand, for $\be > \be_0$ we have
\begin{align}\label{betabig}
\PP\big(|L_n-m| \geq  \be \sqrt{m \log m}\, \big) &< \PP
\big(|L_n-m| \geq
\be_0\sqrt{m\log  m}\,\big)\\
\nonumber &= 4\exp \left(-m \, \frac{c^2}{4(1+c)}\right).
\end{align}
Next, we compare the median and the expectation of $L_n$ using the following
inequality:
\begin{equation*}
| \E L_n - m | \leq   \E |L_n -m |= \int_{0}^{\infty} \;\PP(|L_n-m |
> x) dx.
\end{equation*}
The range of $L_n$ is $[1,n]$, so the integrand is $0$ if $x>n$.
Substitute $x=\be\sqrt{m\log m}$, and divide the integral into two
parts at $\be_0$:
\[
|\E L_n-m| \leq 4\sqrt{m \log m}(I_1+I_2),
\]
where
\begin{equation}\label{I1}
I_1=\int_0^{\be_0}m^{-\be^2/4(1+c)}d\be<\int_0^{\infty}m^{-\be^2/4(1+c)}d\be=
\sqrt{\frac {\pi (1+c)}{\log m}},
\end{equation}
and
\begin{equation}\label{I2}
I_2=\int_{\be_0}^{n/\sqrt{m\log m}} \exp \left(-m \, \frac{c^2}{4(1+c)}\right) d\be
< n \exp \left(-m \, \frac{c^2}{4(1+c)}\right).
\end{equation}
By Lemma \ref{median}, $n < m^4$, so $I_2< m^4\exp (-m \, c^2/4(1+c))$. Since
$m_n$ goes to infinity as $n$ increases (again by Lemma
\ref{median}), the bound on $I_2$ is eventually much smaller than the one on $I_1$:
\begin{align}\label{exp-med}
|\E L_n-m| &\leq 4\sqrt{m\log m}(I_1+I_2) \nonumber \\
&<4\sqrt{\pi (1+c) m}+4\sqrt{m\log m}\,m^4 \exp  \left(-m \,
\frac{c^2}{4(1+c)}\right) \\
&\le 5 \sqrt{\pi (1+c)} \sqrt m \nonumber
\end{align}
for all large enough $n$. Hence $\E L_n$ is of the same order of
magnitude as $m_n$, and we obtain
\begin{equation}\label{medianlim}
\lim n^{-1/3}\E L_n= \lim n^{-1/3} m_n =\al.
\end{equation}
For fixed $\gamma$ and for large enough $n$, (\ref{exp-med}) implies
\begin{align*}
&\PP \big(|L_n -  \E L_n| > \gamma \sqrt{\log n} \;
n^{1/6}\big)\\
&\leq \PP \big(|L_n -  m| > \gamma \sqrt{\log n} \;
n^{1/6}-|\E L_n-m|\, \big)\\
&\leq \PP \big(|L_n -  m| > \gamma \sqrt{\log n} \;
n^{1/6}- 5 \sqrt{\pi (1+c)}\sqrt{m}\, \big).
\end{align*}
Using $m_n \le 3.43 n^{1/3}$ from (\ref{felso}) and
(\ref{medianlim}), it is easy to see that
\begin{align*}
\gamma \sqrt{\log n} \;n^{1/6}- 5 \sqrt{\pi (1+c)\, m} &\ge \g \sqrt{m} \left(
 \sqrt{\frac{3 \log m- \log 41}{3.43}}-\frac { 5 \sqrt{\pi (1+c)}}{\g} \right)\\
&\ge \g \sqrt {\frac {3 }{3.44}} \sqrt {m \log m}.
\end{align*}
Since for large enough $n$, $\g \sqrt {3/3.44} < \beta_0= c
\sqrt{m / \log m}$, (\ref{betasmall}) finally implies
\begin{align*}
&\PP \left(|L_n -  \E L_n| \geq \gamma \sqrt{\log n} \;n^{1/6}\right)\\
&\leq \PP \left(|L_n - m|\ge \g \sqrt {\frac{3} {3.44}}\sqrt{m \log m }\,\right)\\
&\leq 4 m^{- 3 \gamma^2  / 13.76 (1+c) }  < n^{-\gamma^2 / 14,}
\end{align*}
where the last inequality follows from (\ref{medianlim}) and choosing $c=0.01$.
\end{proof}

We remark that the constant in the exponent is far from being the best
possible, and we have made no attempt here to find its optimal value. In
general, Talagrand's inequality is too general to give the precise
concentration, see Talagrand's comments on this in \cite{tal}. Also, we note
that the proof of Theorem~\ref{conc1} also yields the slightly stronger
estimate $n^{- \gamma^2 (1/ 14 + \vartheta})$ for a sufficiently small
$\vartheta$.

For the proof of Theorem~\ref{limsh} we need to consider
subtriangles $S$ of $T$, that is, triangles of the form
$S=\conv\;\{a,b,c\}$ with $a,b,c \in T$, while $X_n$ is still a
random sample from $T$. We will need to estimate the concentration
of the longest convex chain from $X_n$ in $S$. Since this random
variable depends only on the relative area of $S$, we may and do
assume that $T$ is the standard triangle and $S=\conv\{(0,\sqrt
{s}),(0,0),(\sqrt {s},0)\}$. Thus $\A(S)=s/2$. Write $L_{s,n}$ for
the length of the longest convex chain in $S$ from $(0,\sqrt {s})$ to
$(\sqrt {s},0)$, and $m_{s,n}$ for its median. In the
following statements, we consider the situation when $sn/2$, the
expected number of points from $X_n$ in $S$, tends to infinity.

As in the proof of Theorem~\ref{conc1}, we need two estimates: a lower bound
for the median guarantees that the mean and the median are close to each other,
while an upper bound for the expectation (or for the median) is needed to
derive the inequality in terms of $n$. Here comes the lower bound; the case
$s=1$ is Lemma \ref{median}.

\begin{lemma}\label{mediansub} Suppose that $\log (ns) > 25$.
Then $$m_{s,n} \ge \sqrt[3]{3ns/\log (ns)}.$$
\end{lemma}

\begin{proof} Set $t=(\A(S) \log (ns))/(3ns)$, and apply Corollary~\ref{areasum}
to the triangle $S$, resulting in the set of triangles $T_1, \dots T_k$. Then
for the number of triangles we have
\[
\sqrt[3]{3ns/\log (ns)}<k \leq \sqrt[3]{3ns/\log (ns)} +1.
\]
For any $i \in \{1,\dots,k$\}, the probability that $T_i$ contains
no point of $X_n$  is
\begin{align*}
\PP(T_i\cap X_n=\emptyset)&\leq\left(1-\frac {\log
(ns)}{3ns}\right)^n
\\ &< \exp\left(\frac{-\log(ns)}{3s}\right)= (ns)^{-1/3s}<
(ns)^{-1/3}.
\end{align*}
Hence the union bound yields
\begin{align*}
\PP\big(L_{n,s} > \sqrt[3]{3ns/\log (ns)}\,\big) &\ge  1-\PP(T_i\cap
X_n =
\emptyset \mbox{ for some }i \le k) \\
&\ge 1-k\,(ns)^{-1/3}\\
&\ge 1- \big( \sqrt[3]{3/\log (ns)} + (ns)^{-1/3} \big),
\end{align*}
which is greater than $1/2$ by the assumption.
\end{proof}

Obtaining an upper bound for the mean is slightly more delicate; note that in
the lemma below, $s$ need not be fixed.
\begin{lemma}\label{meansub} Assume $ns \to \infty$. Then
$$\lim \, (ns)^{-1/3}\E L_{s,n} =\alpha$$
where $\alpha$ is the same constant as in Theorem~\ref{limit}.
\end{lemma}

\begin{proof} Take any $\eps >0$ and choose $N_0$ (depending on $\eps$)
so large that for every $k \geq N_0$,
$(1-\eps) \alpha < \E L_k \, k^{-1/3} < (1+\eps) \alpha$. The random variable
$K=|X_n\cap S|$
has binomial distribution with mean $ns$. When $ns$ is large enough, $ns - \sqrt{ns
\log ns } \geq N_0$, and we use (\ref{hoeffding}) for a lower estimate:
\begin{align*}
\E L_{s,n} &= \sum_{k=0}^{n} \PP (K=k) \E L_k\\
&\geq \PP(K> ns - \sqrt{ns \log ns}) (1-\eps)\, \alpha \, (ns - \sqrt{ns \log ns})
^{1/3}\\
&\geq (1- (ns) ^{-1/2})(1-\eps) \, \alpha \, (ns - \sqrt{ns \log ns}) ^{1/3}\\
&\geq (1-2 \eps)\,  \alpha \, (ns)^{1/3}.
\end{align*}
For the upper bound, Jensen's inequality applied to
$\sqrt[3]{x}$ comes in handy:
\begin{align*}
\E L_{s,n} &= \sum_{k=0}^{n} \PP (K=k) \E L_k\\
&\leq N_0 \, \PP(K<N_0)
+ \sum_{k=N_0}^{n} \PP (K=k) \E L_k
\\
&\leq N_0 \, + \sum_{k=N_0}^{n} \PP (K=k) \, (1+\eps) \, \alpha \, \sqrt[3]{k}
\\
&\leq N_0 \, + \PP(K \geq N_0) \, (1+\eps) \, \alpha \,\left(
\sum_{k=N_0}^{n} \frac{\PP (K=k)}{\PP(K \geq N_0)} \ k \right)^{1/3}\\
&\leq N_0 \, + \PP(K \geq N_0)^{2/3} \, (1+\eps) \, \alpha \,(
 \E \, K)^{1/3}\\
&\leq N_0 + (1+\eps) \, \alpha \, (ns)^{1/3}\leq (1+2 \eps) \, \alpha \, (ns)^{1/3}.
\qedhere
\end{align*}
\end{proof}

Next, we derive the strong concentration property of $L_{s,n}$, the
analogue of Theorem~\ref{conc1}.

\begin{theorem}\label{concsub}
Suppose $\tau$ is a constant with $0 \leq \tau < 1$. Then for every $\gamma>0$ there
exists a constant
$N$, such that for every $n>N$ and every $s\geq
n^{-\tau}$,
$$\PP \big(|L_{s,n} -  \E L_{s,n}| > \gamma \sqrt{\log ns} \; (ns)^{1/6}\big) <
(ns)^{-\gamma^2/14}.  $$
\end{theorem}

\begin{proof} This proof is almost identical with that of Theorem
\ref{conc1}. Since $L_{s,n}$ is a random variable on $T ^ {\otimes n}$, we can
apply Talagrand's inequality with the certifying function $f(b)=b$ in the same
way as in the proof of Theorem~\ref{conc1}. Write again $m$ for $m_{s,n}$, the
median of $L_{s,n}$. Define $\be_0= c \sqrt{m / \log m}$ with $c=0.01$; then
the estimates (\ref{betasmall}) and (\ref{betabig}) remain valid with $L_{s,n}$
in place of $L_n$. Just as before,
\begin{align*}\label{submedianmean}
| \E L_{s,n} - m| &\leq   \E |L_{s,n} - m |=
\int_{0}^{\infty} \;\PP(|L_{s,n}-m |> x) dx \\
&= 4\sqrt{m \log m}(I_1 + I_2)
\end{align*}
where $I_1$ and $I_2$ are defined the same way as in (\ref{I1})
and (\ref{I2}). Moreover, $I_1$ satisfies the inequality
(\ref{I1}). With $I_2$ we have to be a bit more careful.

Note that $s \geq n^{-\tau}$ with $\tau <1$ guarantees that
Lemma~\ref{mediansub} is applicable for $n > \exp(25/(1-\tau))$. As $x/
\log x$ is monotone increasing for $x > e$,
\[
m \ge \sqrt[3]{\frac{3ns}{\log (ns)}} \ge
\sqrt[3]{\frac{3n^{1-\tau}}{(1-\tau) \log n}} >
\sqrt[3]{\frac{n^{1-\tau}}{ n^{(1-\tau)/2}}} = n^{(1-\tau)/6}
\]
for large enough $n$, and therefore by (\ref{I2})
\[
I_2 < m^{6/(1-\tau)} \exp \left(-m \, \frac{c^2}{4(1+c)}\right)
\]
where of course $6/(1-\tau) <  \infty$. Lemma~\ref{mediansub}
implies that $m = m_{s,n} \rightarrow \infty$, thus  the bound on $I_2$ is much
smaller than the one on $I_1$ for large enough $n$. Therefore, just as in
(\ref{exp-med}),
\begin{align*}
|\E L_{s,n}-m| &\leq 4\sqrt{m\log m}(I_1+I_2)\\
&<4\sqrt{\pi (1+c) m}+4\sqrt{m\log m}\,m^{6/(1-\tau)} \exp  \left(-m \,
\frac{c^2}{4(1+c)}\right) \\
&\le 5 \sqrt{\pi (1+c)} \sqrt m.
\end{align*}
Hence $\E L_{s,n}$ is of the same order of magnitude as $m=m_{s,n}$. Since $sn
\geq n^{1-\tau} \rightarrow \infty$, we can use Lemma~\ref{meansub}, yielding
that for large enough $n$,
\begin{equation}\label{submedfelso}
m_{s,n} \le 3.431 \sqrt[3]{ns}.
\end{equation}
Again for fixed $\gamma$ and for large enough $n$,
\begin{align*}
&\PP (|L_{s,n} -  \E L_{s,n}| > \gamma \sqrt{\log ns} \;
(ns)^{1/6})\\
&\leq \PP (|L_{s,n} -  m| > \gamma \sqrt{\log ns} \;
(ns)^{1/6}-|\E L_{s,n}-m|)\\
&\leq \PP (|L_{s,n} -  m| > \gamma \sqrt{\log ns} \;
(ns)^{1/6}-5 \sqrt{\pi (1+c)} \sqrt m),
\end{align*}
and, by (\ref{submedfelso}),
\begin{align*}
\gamma \sqrt{\log ns} \;
(ns)^{1/6}-5 \sqrt{\pi (1+c)} \sqrt m &\ge \g \sqrt {\frac {3 }{3.44}} \sqrt {m \log
m}.
\end{align*}
Since for large enough $n$, $\g \sqrt {3/3.44} < \beta_0=c
\sqrt{m / \log m}$, (\ref{betasmall}) applied to $L_{s,n}$ and
(\ref{submedfelso}) finally implies
\begin{align*}
&\PP \big(|L_{s,n} -  \E L_{s,n}| \geq \gamma \sqrt{\log ns} \;(ns)^{1/6}\big)\\
&\leq \PP \big(|L_{s,n} - m|\ge \g \sqrt {\frac{3} {3.44}}\sqrt{m \log m }\,\big)\\
&\leq 4 m^{- 3 \gamma^2  / 13.76 (1+c) }  \leq (ns)^{-\gamma^2 /
14}.\qedhere
\end{align*}
\end{proof}

We note that the proof also implies that for any $0<A<B<\infty$, there exists
an $N$ (depending on $A$ and $B$ only), such that the inequality of
Theorem~\ref{concsub} holds for any $\gamma \in [A,B]$ and for every $n>N$.

\section{The conditional approach}\label{condsec}

Our proof of Theorem~\ref{limsh} is based on the following idea. Assume that
$Y$ is a longest convex chain. Recall that $C(Y)$ is its convex polygonal path.
Suppose that $C(Y)$ contains a point $q$ that is far from $\G$, and let $\ell$
be a tangent line of $C(Y)$ at $q$. By Theorem~\ref{mobi+} we know that if
$T_1$ and $T_2$ denote the two triangles determined by $\ell$ and $q$, then
$\sqrt[3]{\A(T_1)}+\sqrt[3]{\A(T_2)}$ is substantially smaller than
$\sqrt[3]{1/2}$. Therefore, if $L^i$ denotes the length of the longest convex
chain in $T_i$, then the expectation of $L^1 + L^2$ is small as well, as it
follows from the strong concentration property of the binomial distribution and
Theorem~\ref{limit}. On the other hand, $C(Y) \subset T_1 \cup T_2$, and hence
$L^1 + L^2$ is at least as large as $L_n$, whose expected value is -- depending
on the choice of the neighbourhood of $\G$ -- much larger than $\E(L^1) +
\E(L^2)$. Therefore, either $L^1$, $L^2$, or $L_n$ is far from its expectation,
which, according to the strong concentration results of the previous sections,
has exponentially small probability.

The technical realisation of the above sketch is not trivial. One is tempted to
proceed in the following way. Define the random variable $Z$ as the indicator
function of the existence of a long convex chain:
\begin{displaymath}
Z = \Bigg\{ \begin{array}{ll}
1 & \textrm{ if $L_n \geq \E L_n - \gamma \sqrt{\log n} \; n^{1/6}$}\\
0 & \textrm{ otherwise.}
\end{array}
\end{displaymath}
Recall that $\cc(X_n)$ is the collection of the longest convex chains in a
given sample. The random variable $Q$ is defined as the (almost surely, unique)
farthest point of $\bigcup \{C(Y) : \ Y \in \cc(X_n)\} $ from $\G$ in Euclidean
distance.

Now, our aim is to show that if $q$ is ``far'' from $\G$, then the conditional
probability $\PP(Z=1\,|\, Q=q)$ -- or, what is the same, the conditional
expectation of $Z$ -- is exponentially small. At first glance, the area
estimate of Theorem~\ref{mobi+} plus the concentration results are sufficiently
strong to derive this statement. However, some thinking reveals that this is
not the case: since the condition on the farthest point modifies the underlying
distribution, the previous results obtained for the uniform sample cannot be
applied.

There are (at least) two ways to correct this reasoning. First, instead of
estimating the conditional expectation in question at every point, we can use a
finite approximation by estimating the (positive) probabilities that a longest
convex chain intersects a small convex set far from $\G$. This method was
accomplished in our article \cite{ab}; it involves several tricks for choosing
the right partitioning of $T$, and it reaches the goal via elementary, but
tedious technical calculations.

For the purpose of  the present thesis, a different method has been developed,
still along the line of conditional probabilities. Note that the independence
of the points of $X_n$ allows us to {\em condition on the location of a point}
of the sample: the remaining $n-1$ points have the same joint distribution as
$X_{n-1}$. Therefore we can estimate the probability of the existence of a long
convex chain through a fixed point, given that the sample contains this point.
This is the motivating  idea.

By default, every result from now on is understood to hold when $n$ is large
enough, where the bound on $n$ depends on $\g$ only. Also, with some ambiguity,
we shall say that a convex chain contains a point, if its polygonal path
contains it -- it will be clear from the context if we require the point to be
a vertex of the path.

Fix the constant $\g\ge 1$, and set
\begin{equation}\label{beq}
b=\g\, n^{1/6}\sqrt{\log n}.
\end{equation}
 We shall call a convex chain $Y
\subset X_n$ {\em long} if its length is at least $\E L_n - b$. The strong
concentration result of Theorem~\ref{conc1} directly shows that the probability
of the existence of long chains is large:
\begin{equation}\label{longexist}
\PP(L_n <\E L_n -b) \le n^{-\g^2/14}.
\end{equation}

\begin{figure}[h]
\epsfxsize =0.45\textwidth \centerline{\epsffile{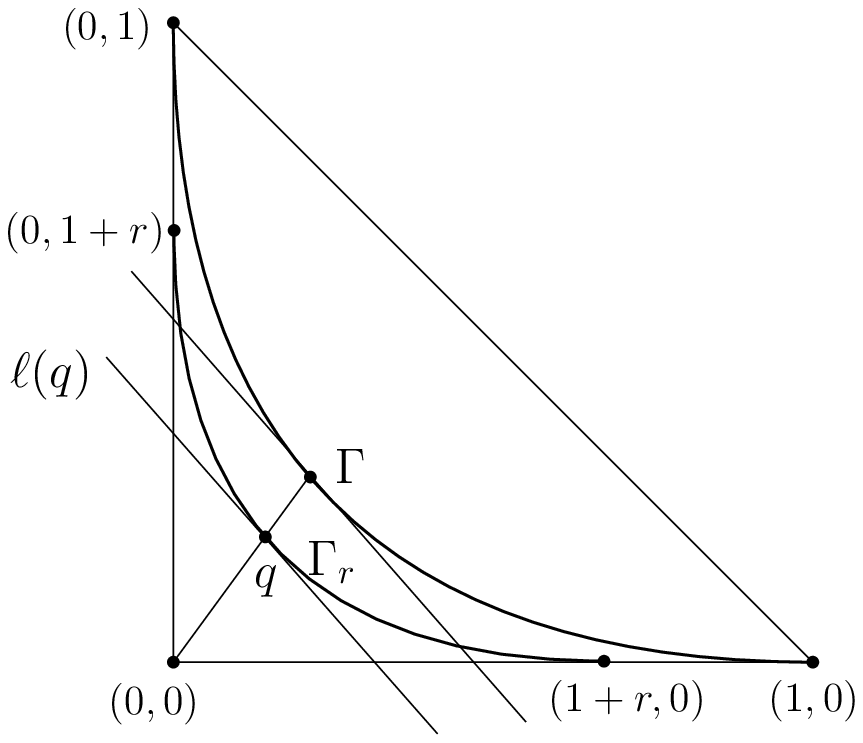}}
  \caption{The parabola $\G_r$}
  \label{parab2}
\end{figure}

\noindent In measuring distances from $\Gamma$ it will be convenient to use the
parabolic arcs
\[
\G_r = \{(x,y) \in T: \sqrt x + \sqrt y=\sqrt{1+r} \},
\]
where $r \in (-1,1)$. Then $\G_0 = \G$, and $\G_r$ is the homothetic copy of
$\G$ with ratio of homothety $1+r$, and center of homothety at the origin, see
Figure~\ref{parab2}. Assume the point $(p,q)$ is on $\G$. Then the point
$\big((1+r)\,p,(1+r)\,q\big)$ is on $\G_r$, and the tangent line to $\G_r$ at
this point is given by the equation
\begin{equation}\label{tanslope}
\frac x{\sqrt p}+\frac y{\sqrt q}=1+r.
\end{equation}
It follows that the distance between parallel tangent lines to $\G$ and $\G_r$
is
\begin{equation}\label{eq:dist}
\frac {|r|}{\sqrt{1/p+1/q}} \le \frac {|r|}{\sqrt 8}.
\end{equation}

Theorem~\ref{limsh} asserts that for a given $\gamma$, the longest convex
chains are within the $\eps$-neighbourhood of $\G$ with probability
$n^{-\gamma^2/14}$, where
 $\eps= 3/2 \gamma^{1/2} \, n^{-1/12} (\log n)^{1/4}$.  Define
\begin{equation}\label{ro}
\rho=\sqrt 8 \, \eps=4 \sqrt{2} \g^{1/2}n^{-1/12}(\log n)^{1/4}.
\end{equation}
Formula (\ref{eq:dist}) immediately implies that if a convex chain $C(Y)$ lies
between $\G_{-\rho}$ and $\G_\rho$, then $\dist(C(Y),\G)\le \eps$. Therefore,
in order to obtain Theorem~\ref{limsh}, we shall prove that all the longest
convex chains are between $\G_{-\rho}$ and $\G_\rho$ with large probability
(meaning that the polygonal paths of the chains are in the required region).

For any point $q \in T$, we define the line $\ell(q)$ as the tangent to $\G_r$
at $q$, where $r$ is the unique parameter such that $q \in \G_r$.

Let $Y$ be a convex chain. By continuity and compactness, the set of $r$'s such
that $C(Y) \cap \G_r \neq \emptyset $ is a closed sub-interval $[r_1, r_2]$ of
$[-1, 1]$. The {\em set of lower extremal points of} $C(Y)$ is defined by
\[
\mathcal{E}_l(Y) =  C(Y) \cap \G_{r_1}\;,
\]
and similarly, the {\em set of upper extremal points of }$C(Y)$ is given by
\[
\mathcal{E}^u(Y) =  C(Y) \cap \G_{r_2}\;.
\]
 We shall simply call elements of $\mathcal E_l(Y) \cup \mathcal E^u (Y) $
  extremal points of $Y$. Plainly, if $q$ is an extremal point, then $\ell(q)$ is a
tangent to $C(Y)$, and $C(Y) \subset T_1 \cup T_2$, where $T_1$ and $T_2$ are
the triangles determined by $q$ and $\ell(q)$, see Figure~\ref{parab2}~a).
There are two cases: if $q$ is a lower extremal point, then by convexity, $q$
must be a vertex of $C(Y)$, that is, $q \in Y$. On the other hand, if $q$ is an
upper extremal point, then there are two points $y_1$ and $y_2$ of $Y$ such
that $q \in y_1 y_2$, the segment between $y_1$ and $y_2$. We shall handle
these cases separately.

For a point $q \in T$ with $q \in \G_\varrho$, we naturally say that $q$ is
{\em below} or {\em above} $\G_r$ depending on whether $\varrho \leq r$ or
$\varrho > r$. For a given $q \in T$, let $\ell(q) \in E(X_n)$ denote the event
that there are two points $p_1, p_2 \in X_n \cap \ell(q)$, for which $q \in p_1
p_2$. We introduce the conditional probabilities $P(q)$ for $q \in T $ by
\begin{displaymath}
P(q) = \Bigg\{ \begin{array}{ll} \PP(\exists \textrm{ long convex chain $Y
\subset X_n$ with $q \in \mathcal{E}_l(Y)$ $|$ $q \in X_n$}) &
\textrm{for $q$ below $\G$;}\\
\PP(\exists \textrm{ long conv. chain $Y \subset X_n$, $q \in \mathcal{E}^u(Y)$
$|$ $\ell(q) \in E(X_n)$}) & \textrm{for $q$ above $\G$.}
\end{array}
\end{displaymath}

Let $S$ denote the subset of $T$ below $\G_{-\rho}$ and $U$ denote the part of
$T$ above $\G_\rho$, where $\rho$ is given by (\ref{ro}). The next theorem
provides a natural and essential link to the conditional probabilities $P(q)$,
and it serves as the key to Theorem~\ref{limsh}.
\begin{theorem}\label{probest}
With the above notations, {\em
\begin{equation}\label{probesteq}
\begin{aligned}
\PP(\exists \textrm{ long convex chain not} & \textrm{ entirely between
$\G_{-\rho}$ and $\G_\rho$})\\
& \leq n  \int_{S} P(q) d \mu (q) + 10 \, n^2  \int_{U} P(q) d \mu (q),
\end{aligned}
\end{equation}
} where $\mu$ stands for the normalised Lebesgue measure on $T$.
\end{theorem}

\begin{proof} This is the moment to use a finite approximation. Let $\delta$ be a small positive
number, and cover $T$ with a disjoint union of squares with axis-parallel sides
of length~$\delta$. The set of the squares in the cover is called $\D$. Now,
the probability on the left hand side of (\ref{probesteq}) is certainly smaller
than
\begin{multline*}
\sum_{D \in \D: D \cap S  \neq \emptyset} \PP(\exists \textrm{ long convex
chain $Y \subset X_n$ with $D \cap  \mathcal{E}_l(Y) \neq \emptyset$ })\\
+\sum_{D \in \D: D \cap U  \neq \emptyset} \PP(\exists \textrm{ long convex
chain $Y \subset X_n$ with $D \cap  \mathcal{E}^u(Y) \neq \emptyset$ })
\end{multline*}

For squares $D \in \D$ intersecting  $S$, the lower extremal point in $D$ is an
element of $X_n$, and hence
\begin{align*}
&\PP(\exists \textrm{ long convex chain $Y \subset X_n$ with } D \cap
\mathcal{E}_l(Y) \neq \emptyset )\\
&=\PP(\exists \textrm{ long convex chain } Y \subset X_n, D \cap
\mathcal{E}_l(Y) \neq \emptyset \ \textrm{and} \ X_n \cap D \neq \emptyset)\\
&=\PP(\exists \textrm{ long convex chain } Y \subset X_n, D \cap
\mathcal{E}_l(Y) \neq \emptyset \ | \ X_n \cap D \neq \emptyset)\ \PP( X_n \cap
D \neq \emptyset).
\end{align*}
Here,
\begin{equation}\label{areaprob}
\PP( X_n \cap D \neq \emptyset) \leq 1 - (1- 2 \delta^2)^n \leq n \,2\,
\delta^2 = n \mu(D),
\end{equation}
and taking limits when $\delta \rightarrow 0$, we obtain the first term of the
right hand side of~(\ref{probesteq}).

Now, let $D \in \D$ intersect $U$. If $D$ contains an upper extremal point of a
long convex chain, then there exists $q \in D$ and $p_1, p_2 \in X_n$, such
that $p_1, p_2 \in l(q)$ and $q \in p_1 p_2$. Let
\[
\ell(D) = \bigcup \{\ell(q) \cap T: q \in D \}
\]
be the ``double cone'' centred at $D$, see Figure~\ref{square1}~a). Then
$\ell(D) \setminus D$ splits into two disjoint parts, that we call ``left'' and
``right'' parts. Let $D_l$ be the union of the left part with $D$, and
similarly, $D_r$ be the union of the right part with $D$. As above,
\begin{align*}
&\PP(\exists \textrm{ long convex chain $Y \subset X_n$ with $D \cap
\mathcal{E}^u(Y) \neq \emptyset$ })\\
&=\PP(\exists \textrm{ long conv. chain $Y \subset X_n$, $D \cap
\mathcal{E}^u(Y) \neq \emptyset$, and }D_l \cap X_n \neq \emptyset,
D_r \cap X_n \neq \emptyset)\\
&\begin{aligned} =\PP(\exists \textrm{ long conv. chain } Y \subset X_n, D \cap
\mathcal{E}^u(Y) \neq \emptyset \ | \ \exists q \in D: \ell(q) \in E(X_n) \\
 \cdot \, \PP(D_l \cap X_n \neq \emptyset, D_r \cap X_n \neq \emptyset). \qquad
\end{aligned}
\end{align*}
By taking limits when $\delta \rightarrow 0$, the correlation of the events
$D_l \cap X_n \neq \emptyset$ and $D_r \cap X_n \neq \emptyset$ tends to 0.
Therefore, the proof will be completed as above if we show that
\[
\PP(D_l \cap X_n \neq \emptyset)\; \PP(D_r \cap X_n \neq \emptyset) \leq 20 \,
n^2 \delta^2
\]
for every $D \in \D$ intersecting $U$. Similarly to (\ref{areaprob}),
\[
\PP(D_l \cap X_n \neq \emptyset)\; \PP(D_r \cap X_n \neq \emptyset) \leq 4 n^2
A(D_l) A(D_r),
\]
and hence the inequality to prove is
\begin{equation}\label{ad}
A(D_l) A(D_r) \leq 5 \, \delta^2\,.
\end{equation}

\begin{figure}[h]
\epsfxsize = \textwidth \centerline{\epsffile{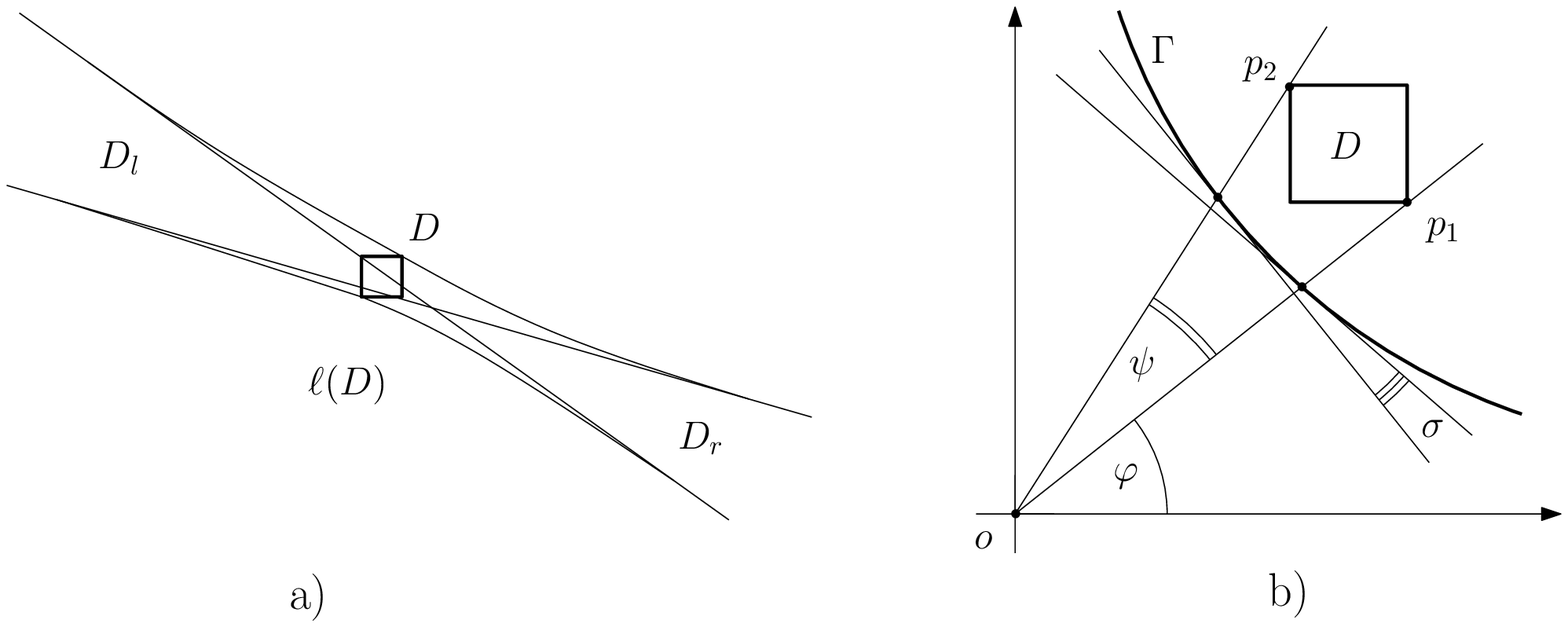}}
  \caption{Squares above $\G$}
  \label{square1}
\end{figure}

 Let the vertices of $D$ in counter-clockwise order be $(x,y)$,
$(x, y+\delta)$, $(x-\delta, y+\delta)$ and $(x-\delta, y)$. It is easy to see
that among the lines $\ell(q)$ with $q \in D$, the ones  with extremal slopes
are the ones belonging to $(x,y)$ and $(x-\delta, y+\delta)$. Let us call $p_1
= (x,y)$, the lower right vertex,  and $p_2=(x-\delta, y+\delta)$, the upper
left vertex of $D$, see Figure~\ref{square1}~b). Let $\varphi$ be the slope of
the ray $o \, p_1$, where $o$ is the origin. By symmetry, we may assume that
$\varphi \leq \pi/2$.

Denote  $\psi$ the angle between the rays $o \, p_1$ and $o \, p_2$.  If
$\delta$ is small enough, then both $p_1$ and $p_2$ are above $\G$, and hence
their distance from $o$ is at least $1/(2 \sqrt 2$). On the other hand, $|p_1
p_2| = \sqrt{2} \delta$. Therefore
\begin{equation}\label{psi}
\psi \leq \arcsin 4 \delta \approx 4 \delta.
\end{equation}

By formula (\ref{tanslope}), for any point $p$ on the ray $o\, p_1$, the lines
$\ell(p)$ and $\ell(p_1)$ are parallel, and their slope is
\[
\nu(\varphi) =- \arctan(\sqrt{\tan \varphi}).
\]
The slope of $\ell(p_2)$ is $\nu(\varphi + \psi)$.
 Now,
\[
\nu\,'(\varphi) = - \frac{1}{2(\sin \varphi + \cos \varphi)\sqrt{\sin \varphi
\cos \varphi}}\,.
\]
As it can easily be checked,  $\nu\,'(\varphi)$ is a concave function, and
$|\nu\,'(\varphi)|\leq 1/\sqrt{2 \sin 2 \varphi}$. Therefore the angle between
$\ell(p_1)$ and $\ell(p_2)$ is
\begin{equation}\label{sigma}
\sigma =  \nu(\varphi) - \nu(\varphi + \psi) \leq \frac{\psi}{\sqrt{2 \sin 2
\varphi}}\;.
\end{equation}

\begin{figure}[h]
\epsfxsize = 0.45 \textwidth \centerline{\epsffile{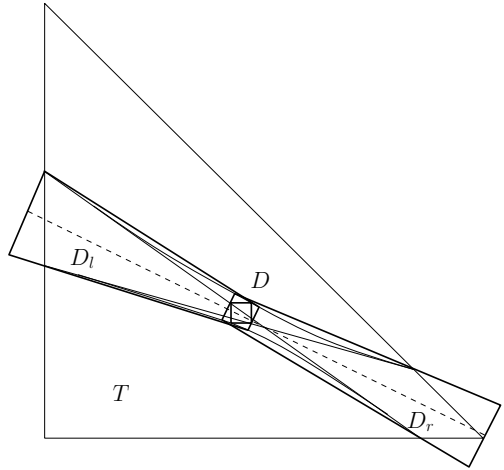}}
  \caption{Covering $\ell(D)$ with a bow-tie}
  \label{square2}
\end{figure}

We divide the rest of the argument into two parts depending on the size of
$\varphi$. If $\varphi > 0.17$, then $1/\sqrt{2 \sin 2 \varphi} < 1.23
<\sqrt{5} -1 $, and hence by (\ref{psi}),
\[
\sigma \leq 4 (\sqrt{5}-1)\,\delta.
\]
Therefore $D_l \cup D_r$ can be covered with a ``bow-tie'': the union of two
oppositely placed trapezoids with shorter base length at most $\sqrt{2}
\delta$, whose angle between the non-parallel opposite sides is less than $4
(\sqrt{5}-1) \, \delta$, and the sum of their altitudes is at most $\sqrt{2}$,
see Figure~\ref{square2}. A straightforward calculation reveals that under
these conditions, $A(D_l) A(D_r)$ is maximal when both the altitudes are
$1/\sqrt{2}$, and then
\[
A(D_l) A(D_r) < 5 \, \delta^2.
\]

Finally, when $\varphi \leq 0.17$, then $2 \varphi / \sin 2 \varphi \leq 1.02$,
and hence (\ref{psi}) and (\ref{sigma})  essentially imply that
\[\sigma \leq \frac{2 \delta }{ \sqrt
\varphi}\;.
\]
We shall justify the use of this estimate by leaving a sufficient gap at the
end. Recall that $x$ and $y$ are the coordinates of the vertex $p_1$. Since
$p_1$ is above $\G$, they satisfy the inequality $\sqrt{x} + \sqrt{y}>1$. On
the other hand, $\tan \varphi  = y/x$. These relations immediately yield that
\[
x  \geq \frac{\cos \varphi} {\sqrt{\cos \varphi}+ \sqrt{\sin \varphi}} \approx
\frac{1}{1 + \sqrt{\varphi}} \; .
\]
Moreover, the slopes of the tangent lines $\ell(p_1)$ and $\ell(p_2)$ are at
most $0.4$. Therefore, $D_r$ is covered by a trapezoid, whose shorter base is
of length at most $1.1 \delta $, the angle between its non-parallel opposite
edges is at most $2 \delta / \sqrt \varphi$, and its altitude is at most
\[
1.1 \left( 1- \frac{1}{1 + \sqrt{\varphi}} \right) \leq 1.1 \sqrt{\varphi}.
\]
Hence, $A(D_r) \leq 2.5\, \delta \sqrt{\varphi}$. On the other hand, $D_l$ is
covered by a trapezoid with shorter base-length $1.1 \delta $, opening angle at
most $2 \delta / \sqrt \varphi$ and altitude at most $1.1$, therefore
\[
A(D_l) \leq 1.21 \, \delta \left(1+ \frac{1}{\sqrt{\varphi}} \right).
\]
Altogether,
\[
A(D_r) A(D_l) \leq 3.025 (1 + \sqrt{\varphi})\, \delta^2 < 4.3 \, \delta^2,
\]
therefore (\ref{ad}) holds for sufficiently small $\delta$, and the proof is
complete.
\end{proof}

\section{Limit shape}\label{limshsec}

In this section we finally prove Theorem~\ref{limsh}.
 The first lemma  formalises the intuitive idea that was presented
at the beginning of Section~\ref{condsec}. Let $q \in T$ be an arbitrary point,
and as usual, let $T_1$ and $T_2$ denote the two triangles determined by
$\ell(q)$ and $q$. Let $X_n$ be a random sample of $n$ points from $T$ and let
$L^i$ denote the length of the longest convex chain in $T_i$, $i=1,2$.

\begin{lemma}\label{triexp} For sufficiently large $n$, if $|r|\ge n^{-1/12}$, then
\[
\E L^1+\E L^2 \le \E L_n -0.52 \,r^2\sqrt[3]{n}.
\]
\end{lemma}

\begin{proof} Let $t_i = \mu (T_i)$ for $i=1,2$. We want to apply
Theorem~\ref{mobi+}. It is not hard to see (using Corollary~\ref{erinto} for
instance) that what is denoted by $|a-b|$ there, is equal to $|r|$ here.
Consequently
\begin{equation}\label{terhiany}
 \sqrt[3]{t_1 /2} + \sqrt[3]{t_2 /2} \leq  \sqrt[3]{1/2} -\sqrt[3]{1/2} \; \frac 1 3
\, r^2.
\end{equation}

Write $L^i$ for the longest convex chain in the triangle $T_i$. By affine
invariance, $L^i$ has the same distribution as $L_{t_i,n}$ (from
Section~\ref{concentr}) for $i=1,2$. We need to estimate $\E L_n- (\E L^1 + \E
L^2)$ from below.

\begin{figure}[h]
\epsfxsize =0.4 \textwidth \centerline{\epsffile{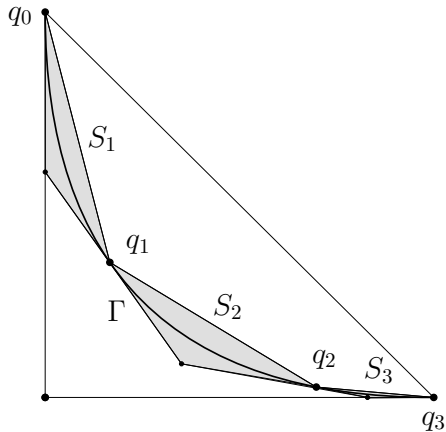}}
  \caption{Estimating the expectations}
  \label{extratri}
\end{figure}

For four points $q_0=(0,1)$, $q_1$, $q_2$ and $q_3=(1,0)$ in this order on
$\G$, denote $S_i$ the triangle delimited by the tangents to $\G$ at
$q_{i-1},q_i$, and by the segment $q_{i-1} q_i$, $i=1,2,3$; see
Figure~\ref{extratri}. Choose $q_1$ and $q_2$ so that $\A(S_1)=t_1/2$ and
$\A(S_2)=t_2/2$. Then Corollary~\ref{areasum} and (\ref{terhiany}) imply that
\begin{equation*}
\sqrt[3]{ \A (S_3)} \geq \sqrt[3]{1/2}\; \frac 1 3 \, r^2.
\end{equation*}
Let now $\La^i$ denote the length of a longest chain in $S_i$ for $i=1,2,3$.
For $i=1$ and $2$, $\La^i$ has the same distribution as $L_{t_i,n}$ (and as
$L^i$). Therefore $\E L^i = \E L_{t_i,n}= \E \La^i$ for $i=1,2$. Further,
$\La^1+\La^2+\La^3 \le L_n$ follows from concatenating the longest convex
chains in the triangles $S_i$. Thus we have
\begin{equation}\label{osszegkicsi}
\E L^1 + \E L^2 +\E \La^3 = \sum_{i=1}^{3}\E \La^i \leq  \E L_n.
\end{equation}
The random variable $|X_n\cap S_3|$ has binomial distribution with mean
$2A(S_3)n$ which is at least $\kappa = (1/3)^3 r^6 n \ge (1/3)^3n^{1/2}$. Set
$N=\kappa - \sqrt {\kappa \log \kappa}$. Thus we obtain that for all large
enough $n$,
$$N > 0.99 \, \kappa = \frac{0.99}{27} r^6 n,$$
and $N$ tends to infinity with $n$. Using the estimates (\ref{hoeffding}) and
(\ref{also}), again for large $n$ we have
\begin{align*}
 \E \La^3  &\geq \PP (|X_n\cap S_3| \geq N) \, \E L_N \geq (1-\kappa^{-1/2}) \ 1.57
\, N^{1/3}\\
&\geq 1.569 N^{1/3} \geq 0.52\, r^2 \sqrt[3]n.
\end{align*}
Hence, by (\ref{osszegkicsi})
\begin{equation*}
\E L^1 + \E L^2 \leq \E L_n- 0.52 \,r^2\sqrt[3]n. \qedhere
\end{equation*}
\end{proof}

Next, we estimate the conditional probabilities $P(q)$ from
Section~\ref{condsec} for points $q$ that are far from~$\G$.

\begin{lemma}\label{cond}
For any fixed $\gamma>0$, there exists an $N$, such that for every $n >N$,
\[
P(q) \leq n^{-31 \gamma^2 /14}
\]
for every $q$, that is below $\G_{-\rho}$ or above $\G_\rho$.
\end{lemma}

\begin{proof}
Assume that $Y$ is a long convex chain that contains $q$. If $q$ is below
$\G_{-\rho}$, then $X_n \setminus q$ is distributed as $X_{n-1}$. Therefore if
$\tilde{L}^1$ and $\tilde{L}^2$ denote the length of the longest convex chains
in $T_1$ and $T_2$, then $\tilde{L}^i$ is distributed as $L_{t_i, n-1}$ for
$i=1,2$, where $t_i = \mu (T_i)$. Moreover, $|Y \cap T_i| \leq \tilde{L}^i$.
Recall that, since $Y$ is a long convex chain, its length is at least $\E L_n -
b$, where $b$ is given by~(\ref{beq}). Thus,
\[
\E L_n - b \leq |Y| \leq |Y \cap T_1| + |Y \cap T_2| \leq \tilde{L}^1 +
\tilde{L}^2 + 1.
\]
Therefore,
\[
P(q) \leq \PP(\tilde{L}^1 + \tilde{L}^2 + 1 \geq \E L_n - b).
\]
Since $n,b \rightarrow \infty$, the term ``$+1$'' makes no difference at the
estimates. Furthermore, $\PP(\tilde{L}^1 + \tilde{L}^2  \geq \E L_n - b) \leq
\PP(L^1 + L^2 \geq \E L_n - b)$, and hence
\begin{equation}\label{plest}
P(q) \leq \PP(L^1 + L^2  \geq \E L_n - b).
\end{equation}
When $q$ is above $\G_\rho$, then there are two points $y_1, y_2 \in X_n$ on
$l(q)$ such that $q \in y_1 y_2$, and $Y$ is contained in the triangles
$\tilde{T}_1$ and $\tilde{T}_2$ determined by $\ell(q)$, $y_1$ and $y_2$. Now,
$X_n \setminus \{y_1, y_2 \}$ is distributed as $X_{n-2}$, and $\tilde{T}_1
\subset T_1$, $\tilde{T}_2 \subset T_2$. Therefore a similar reasoning as above
results in (\ref{plest}).

Lemma~\ref{triexp}, (\ref{beq}) and (\ref{ro}) gives that
\[
\E L^1 + \E L^2 \leq \E L_n- 0.52 \, \rho^2 n^{1/3} < \E L_n - 13 b,
\]
therefore
\begin{equation}\label{psum}
\begin{aligned}
P(q)&\leq \PP(L^1 + L^2  \geq \E L_n - b)\\
&\leq \PP(L^1 + L^2  \geq \E L^1 + \E L^2 + 12 b)\\
&\leq \sum_{i=1,2} \PP(L^i\ge \E L^i+6b).
\end{aligned}
\end{equation}

 Next, we estimate $\PP(L^i\ge \E L^i+6b)$. When $t_i = 2 A(T_i) \ge
n^{-5/6}$, we use Theorem~\ref{concsub} with $\tau=5/6$:
\begin{align*}
& \PP(L^i\ge \E L^i+6b) = \PP(L^i \geq \E  L^i +
6 \gamma \sqrt{\log n} \; n^{1/6} )\\
&\leq \PP(L^i \geq \E L^i + 6\gamma \sqrt{\log n / \log (n t_i)} \, \sqrt{\log
(n t_i)} \; (n t_i)^{1/6} )
\\
&\leq (n t_i)^{-\gamma^2 36\log n / 14 \log (n t_i)}= n^{-36\gamma^2/14}.
\end{align*}
The last inequality holds according to the remark following
Theorem~\ref{concsub}, since
\[
1 \leq 6 \gamma \sqrt{\log n / \log (n t_i)} \leq \gamma \, 6^{3/2}.
\]
Finally, when $t_i < n^{-5/6}$, the expected number of points in $T_i$ is $t_i
n < n^{1/6}$. So for the random variable $|T_i\cap X_n|$ inequality
(\ref{hoeffdingfelso}) implies that
\begin{align*}
\PP \big(\, |T_i\cap X_n| \geq  6 \gamma \sqrt{\log n} \; n^{1/6} \big)&\leq
\left(\frac{e \, t_i n}{6 \gamma \sqrt{\log n} \; n^{1/6}}   \right)^{6 \gamma
\sqrt{\log n} \; n^{1/6}}\\
\leq \left(\frac{e }{6 \gamma \sqrt{\log n}}   \right)^{n^{1/6}}&<
n^{-32\gamma^2/14}
\end{align*}
for large enough $n$, as it can easily be seen by taking logarithms.

Therefore, in both cases
\[
\PP \big(L^i \geq \E L^i + 6 \,b \big) < n^{-32\gamma^2/14},
\]
and by (\ref{psum}),
\[
P(q) < 2 n^{-32\gamma^2/14} < n^{-31\gamma^2/14}. \qedhere
\]
\end{proof}

\begin{proof}[Proof  of Theorem~\ref{limsh}.]
We have to estimate the probability that there is a {\em longest} convex chain
not lying between $\G_{-\rho}$ and  $\G_{\rho}$. This event splits into two
parts: either the longest convex chain has less than $\E L_n - b$ points,  or
there is a long convex chain not entirely between $\G_{-\rho}$ and $\G_{\rho}$.
By Theorem~\ref{conc1} and the remark following it,
\[
\PP(L_n < \E L_n - b)< n^{-\g^2(1/14 + \vartheta)}
\]
 for some positive $\vartheta >0$. On the other
hand, Theorem~\ref{probest}, Lemma~\ref{cond} and the condition $\gamma \geq 1$
imply that
\begin{align*}
&\PP(\exists \textrm{ long convex chain not entirely} \textrm{ between
$\G_{-\rho}$ and
$\G_\rho$})\\
 &\leq 10 \, n^2 \, n^{-31\gamma^2/14}
\leq n^{-2\gamma^2/14}.
\end{align*}
Therefore the probability in question is at most
\begin{equation*}
n^{-\g^2/14} n^{-\g^2 \vartheta}+  n^{-2 \g^2/14} = n^{-\g^2/14}\left(n^{-\g^2
\vartheta} + n^{-\g^2/14}\right)<n^{-\g^2/14}\;. \qedhere
\end{equation*}
\end{proof}

\section{Numerical experiments}\label{experi}

In the final section we summarize the observations obtained by
computer simulations.

The search for the longest convex chains can be accomplished by an algorithm
which has running time $O(n^2)$. This algorithm works as follows. We order the
points by increasing $x$ coordinate, and then recursively create a list at each
point. The $k$th element on the list at point $p$ contains the minimal slope of
the last segment of chains starting at $p_0$ and ending at $p$ whose length is
exactly $k$, and a pointer to the other endpoint of this last segment. For
creating the list at the next point $q$, we have to search the points before
$q$, and check if $q$ can be added to their minimal slope chains while
preserving convexity.

This algorithm can be speeded up with some (not fully justified, but useful)
tricks. First of all, Theorem~\ref{limsh} guarantees that we have to search
only among the points close to $\G$. The simulations show that most longest
convex chains are located in a small neighbourhood of $\G$, whose radius is in
fact of order approximately $n^{-1/3}$, much smaller than the width of order
$n^{-1/12}$ given by Theorem~\ref{limsh}. Therefore the search can be
restricted to a subset of the points with cardinality of order $n^{2/3}$.
Second, when looking for the longest chain, we have to search only points
relatively close to $p$, and chains which are already relatively long, thus
reducing memory demands.

\begin{figure}[h]
\epsfxsize =0.4\textwidth \centerline{\epsffile{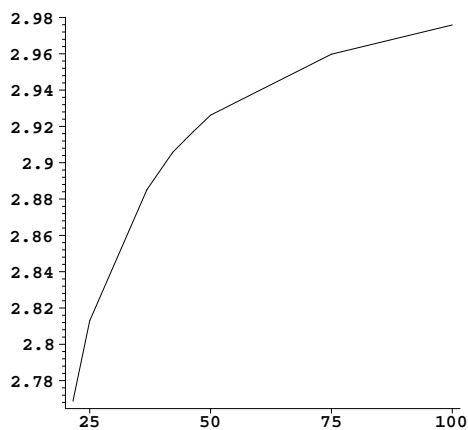}}
  \caption{Results for $n^{-1/3} \E L_n$, illustrated as a function of $n^{1/3}$. }
  \label{varh}
\end{figure}

With the above method, the search can be executed for up to $5\cdot 10^4$
active points, in which case examining one sample takes about 2 minutes. As the
experiments show, this provides a good approximation for $n$'s up to order
$10^6$. In each experiment, we increased the width of the searched
neighbourhood until the increment did not generate a significant change in the
average length of the longest convex chain. The results obtained by this
method, although giving only a lower bound for $\E L_n$, are heuristically
close to it.

Our largest search was done for $n= 10^6$. The number of samples was $250$
except for the cases $n=25^3$ and $n=10^6$, where we used $500$ samples in
order to model the distribution of $L_n$ (see Figure~\ref{eloszl}).

\begin{table}[h] \label{tabla}
\begin{center}
\newcommand\TT{\rule{0pt}{2.6ex}}
\newcommand\BB{\rule[-1.2ex]{0pt}{0pt}}
\begin{tabular}{c| c | c | c | c }
$n$ & $n^{-1/3} \E L_n$ \BB & $d_n$ &Distance$ /\sqrt{2}$ & Deviation
\\
\hline
1000 \TT & 2.532 & 4 & 0.270 & 1.254\\
10000& 2.768 & 5 & 0.200 & 1.383\\
15625& 2.813 & 5 & 0.150 & 1.293\\
50000& 2.885 & 5 & 0.100 & 1.411\\
75000& 2.906 & 5 & 0.070 & 1.580\\
100000& 2.917& 5 & 0.060 & 1.431\\
125000& 2.926 & 5 & 0.050 & 1.637\\
421875& 2.959 & 5 & 0.012 & 1.732\\
1000000 \BB & 2.976 & 6 & 0.012 & 2.023\\
\end{tabular}
\\[10pt]
\caption{Results obtained by the simulation}
\end{center}
\end{table}

The obtained numerical results well illustrate what the proof of
Theorem~\ref{limit} suggests, namely, that $ n^{-1/3} \E L_n$ is increasing
with $n$. Also, the data seem to confirm that $\alpha=3$.

On Table~\ref{tabla} we list the results obtained by the program. The first
column is the number of points chosen in $T$, the second is the average of
$n^{-1/3}L_n$. The third column contains the half-length of the interval of the
values of $L_n$, that is, $d_n= \lfloor \max |L_n - \E L_n|\rfloor $. This is
noticeably small even for $n=10^6$. In the fourth column we list $1/\sqrt{2}$
times the radius of the neighbourhood of parabola we used for the search (the
term $\sqrt{2}$ comes from a transformation of coordinates). The last data are
the standard deviation of the set of values of $L_n$, ie. the square-root of
its variance.

Figure~\ref{varh} illustrates the linear interpolation of $n^{-1/3} \E L_n $ as
a function of $n^{1/3}$. It is based on the data shown on Table~\ref{tabla}.

As we know from Theorem~\ref{conc1}, $L_n$ is highly concentrated about its
expectation. This phenomenon is well recognizable on Figure~\ref{eloszl}, where
we plot the distribution in the cases $n=25^3(=15625)$ and $n=10^6$ with $500$
samples.
\bigskip

 \begin{figure}[h]
\centering
\includegraphics[width=0.4 \textwidth]{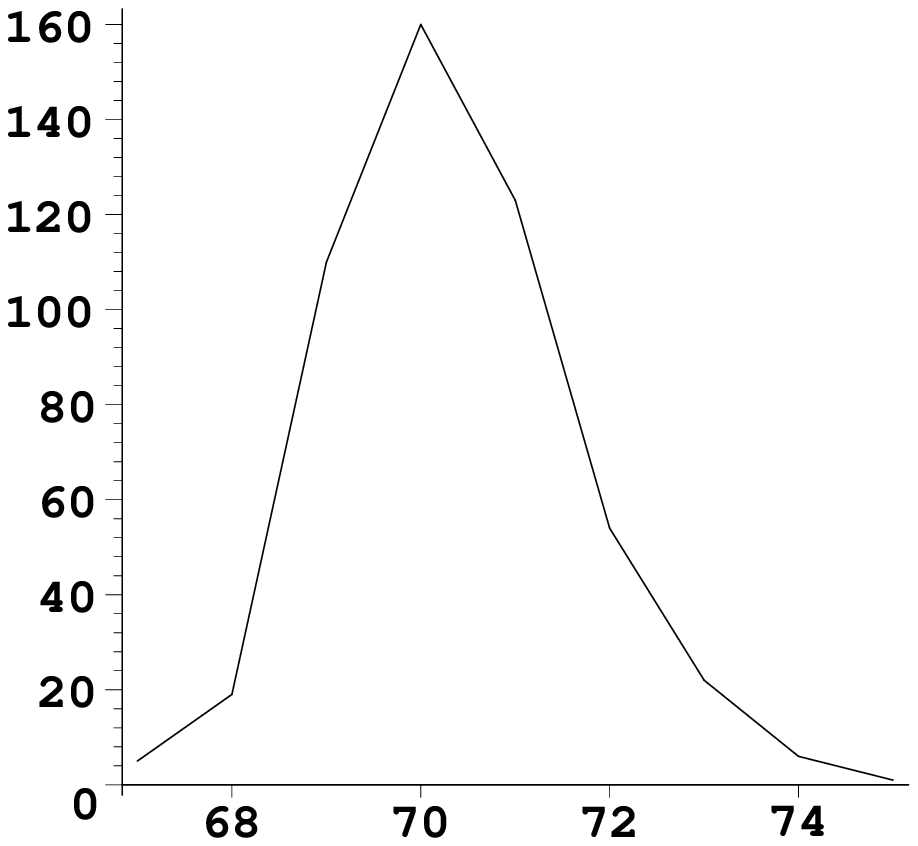}
\hspace{1 cm}
\includegraphics[width=0.4 \textwidth]{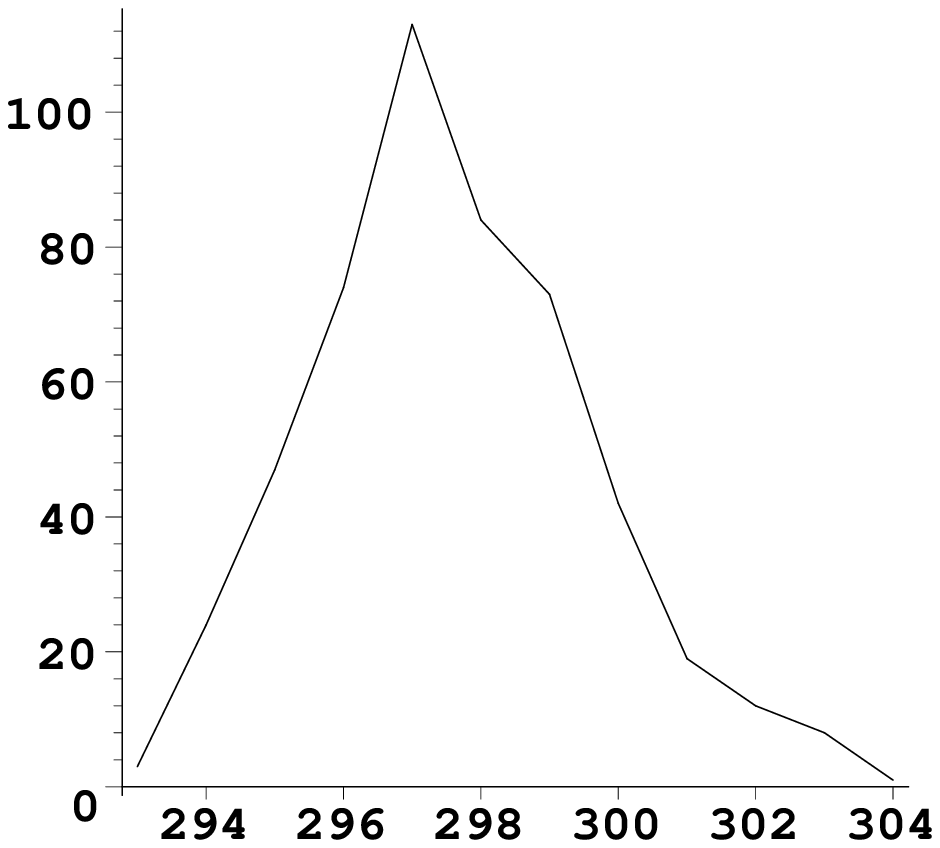}
\caption{Distribution of  $L_n$, 500 samples, $n=25^3$ and $n=10^6$.}
\label{eloszl}
\end{figure}


\end{document}